\documentclass[12pt,reqno,a4paper]{amsart}

\usepackage{amsmath,amssymb,amsthm,amsfonts} %Math symbols
\usepackage{mathtools} %Math symbols
\usepackage{thmtools} %Necessary for theorem restate
\usepackage{thm-restate} % Theorem restate
\usepackage{enumitem} %Lists
\usepackage{dynkin-diagrams} %Dynkin diagrams
\usepackage{caption} %Captions of floats like tables, figures,...
\usepackage{xcolor} %Colours
\usepackage{makecell} %Cells in Tables
\usepackage{adjustbox} % Scale, resize, rotate, trim and/or clip boxes
\usepackage{hyperref} % Hyperlinks
\usepackage[capitalise]{cleveref} %References
\usepackage{tikz-cd} %Diagrams

\usepackage{fouriernc}
\usepackage[T1]{fontenc}

\numberwithin{equation}{section}

\newcommand{\GL}{\operatorname{GL}}
\newcommand{\PGL}{\operatorname{PGL}}

\newcommand{\SL}{\operatorname{SL}}
\newcommand{\PSL}{\operatorname{PSL}}

\newcommand{\Sp}{\operatorname{Sp}}

\newcommand{\End}{\operatorname{End}}
\newcommand{\Aut}{\operatorname{Aut}}

\renewcommand{\sl}{\operatorname{\mathfrak{sl}}}
\newcommand{\so}{\operatorname{\mathfrak{so}}}
\newcommand{\su}{\operatorname{\mathfrak{su}}}
\renewcommand{\sp}{\operatorname{\mathfrak{sp}}}

\DeclareMathOperator{\Hom}{Hom}
\DeclareMathOperator{\Id}{Id}
\DeclareMathOperator{\hol}{hol}
\DeclareMathOperator{\Fix}{Fix}
\DeclareMathOperator{\Is}{Is}

\newcommand{\R}{\mathbb{R}}
\newcommand{\C}{\mathbb{C}}
\renewcommand{\H}{\mathbb H}

\newcommand{\cB}{\mathcal{B}}

\newcommand{\cG}{\mathcal{G}}

\newcommand{\bK}{\mathbb{K}}

\newcommand{\goe}{\mathfrak{e}}
\newcommand{\gof}{\mathfrak{f}}
\newcommand{\gog}{\mathfrak{g}}
\newcommand{\goh}{\mathfrak{h}}
\newcommand{\gok}{\mathfrak{k}}

\newcommand\subsetsim{\mathrel{%
\ooalign{\raise0.2ex\hbox{$\subset$}\cr\hidewidth\raise-0.8ex\hbox{\scalebox{0.9}{$\sim$}}\hidewidth\cr}}}

\newcommand{\invol}[2]{\draw[latex-latex] (root #1) to [out=-60,in=-120] (root #2);}
\newcommand{\invold}[2]{\draw[latex-latex] (root #1) to [out=-60,in=60] (root #2);}

\newtheorem{theorem}{Theorem}[section]
\newtheorem{corollary}[theorem]{Corollary}

\newtheorem{lemma}[theorem]{Lemma}

\theoremstyle{definition}
\newtheorem{definition}[theorem]{Definition}

\newtheorem{remark}[theorem]{Remark}
\newtheorem{example}[theorem]{Example}

\newlist{enumthm}{enumerate}{1} % set up a dedicated enumeration env.
\setlist[enumthm]{label=\upshape(\alph*),ref=\upshape\thetheorem(\alph*)}
\crefalias{enumthmi}{theorem} % alias 'enumthmi' counter to 'thm'

\newlist{enumprop}{enumerate}{1} % set up a second dedicated enumeration env.
\setlist[enumprop]{label=\upshape(\alph*),ref=\upshape\theproposition(\alph*)}
\crefalias{enumpropi}{proposition} % alias 'enumpropi' counter to 'cor'

\newlist{enumlem}{enumerate}{1} % set up a third dedicated enumeration env.
\setlist[enumlem]{label=\upshape(\alph*),ref=\upshape\thelemma(\alph*)}
\crefalias{enumlemi}{lemma} % alias 'enumlemi' counter to 'lem'

\newlist{enumcor}{enumerate}{1} % set up a fourth dedicated enumeration env.
\setlist[enumcor]{label=\upshape(\alph*),ref=\upshape\thecorollary(\alph*)}
\crefalias{enumcori}{corollary} % alias 'enumcori' counter to 'cor'

\newlist{enumdef}{enumerate}{1} % set up a fifth dedicated enumeration env.
\setlist[enumdef]{label=\upshape(\alph*),ref=\upshape\thedefinition(\alph*)}
\crefalias{enumdefi}{definition} % alias 'enumdefi' counter to 'def'

\author{Anton Hase}
\address{Mathematics Department, Technion, Haifa 320003, Israel}
\email{hase@campus.technion.ac.il}
\title{Integrability of quaternion-K\"ahler symmetric spaces}

\begin{document}

\begin{abstract}
We find a necessary condition for the existence of an action of a Lie group $G$ by quaternionic automorphisms on an integrable quaternionic manifold in terms of representations of $\gog$. We check this condition and prove that a Riemannian symmetric space of dimension $4n$ for $n\geq 2$ has an invariant integrable almost quaternionic structure if and only if it is quaternionic vector space, quaternionic hyperbolic space or quaternionic projective space.
\end{abstract}

\maketitle

\section{Introduction}\label{sec:introduction}

Hermitian symmetric spaces are well-studied objects at the crossroads of Lie theory, differential geometry and complex analysis. A Hermitian symmetric space $M^{2n}=G/K$ is a Riemannian symmetric space together with a $G$-invariant almost complex structure. One can show that any such $M$ is K\"ahler, and hence the $G$-invariant almost complex structure is automatically integrable. Hermitian symmetric spaces of non-compact type can even be realized as bounded domains in complex vector spaces as shown by Borel and Harish-Chandra. (For an overview on Hermitian symmetric spaces see Chapter VIII in \cite{Helgason2001}.) 
 
If we replace the complex numbers $\C$ by the division algebra of quaternions $\H$ in the definition of Hermitian symmetric spaces, then we obtain Riemannian symmetric spaces $M^{4n}=G/K$ with $G$-invariant almost quaternionic structures. For $n\geq 2$, one can still characterize these spaces by their holonomy.

\begin{restatable*}{proposition}{almostqsymmetricimpliesqK}\label{prop:almostqsymmetricimpliesqK}
A Riemannian symmetric space $G/K$ of dimension $4n$ for $n\geq 2$ has a $G$-invariant almost quaternionic structure if and only if it is quaternion-K\"ahler.
\end{restatable*} 
While quaternion-K\"ahler symmetric spaces emerge in an analogous way to Hermitian symmetric spaces, the analogy does not reach very far as far as integrability is concerned. We will show the following theorem: \begin{restatable*}{theorem}{symmetricintegrable}\label{thm:symmetricintegrable}
A quaternion-K\"ahler symmetric space is integrable if and only if it is quaternionic vector space $\H^n$, quaternionic hyperbolic space $H^n_\H$ or quaternionic projective space $\mathbb{P}(\H^{n+1})$.
\end{restatable*} This was previously only known for compact quaternion-K\"ahler symmetric spaces (\cref{cor:compactcase}). This proof can not be extended to the non-compact case, since it uses topological obstructions which vanish for all non-compact quaternion-K\"ahler symmetric spaces.

One approach to decide the question of integrability for almost quaternionic manifolds is the following: Cartan's equivalence method provides structural invariants of a $\cG$-structure of finite type that determine the $\cG$-structure. For almost quaternionic manifolds this boils down to the torsion of the almost quaternionic structure and an analogue of the Weyl tensor, see \cite{Salamon1986}. We will follow a different path here. Instead of using the Riemannian structure (or a torsion-free connection), we will focus on the invariance of the almost quaternionic structure on a quaternion-K\"ahler symmetric space. Our approach has the advantage, that we can treat other homogeneous spaces at the same time. For example, we will also show: \begin{restatable*}{corollary}{exceptional}\label{cor:exceptional}
No real form of $E_6$, $E_8$, $F_4$ or $G_2$ can act transitively by quaternionic morphisms on a connected integrable quaternionic manifold of dimension $n\geq 1$.
\end{restatable*} The situation for $E_7$ is more complicated, but we can still show: \begin{restatable*}{corollary}{affineclassical}\label{cor:affineclassical}
If an irreducible affine symmetric space $G/H$ admits an invariant integrable almost quaternionic structure, then $G$ is of classical type.
\end{restatable*}

%This gives us one way of proving \cref{thm:symmetricintegrable}, but Instead of trying to compute this tensor for every non-compact quaternion-K\"ahler symmetric space, 

This article is organized as follows: In \cref{sec:preliminaries} we will review the definitions of different quaternionic structures. Then we will see which of them apply to Riemannian symmetric spaces. This will lead us to the classification of quaternion-K\"ahler symmetric spaces by Wolf (see \cite{Wolf1965}), which we will give in \cref{table:noncompqKsymspace}.

In \cref{sec:Kulkarni} we will pass from integrable quaternionic manifolds to manifolds with a $(\PGL_{n+1}(\H),\mathbb{P}(\H^{n+1}))$-structure using a theorem of Kulkarni:\begin{restatable*}[Theorem 3.4. in \cite{Kulkarni1978}]{theorem}{kulkarni}\label{thm:kulkarniquaternion}
An almost quaternionic structure on a smooth manifold $M^{4n}$ is integrable if and only if it is compatible with a $(\PGL_{n+1}(\H),\mathbb{P}(\H^{n+1}))$-structure on $M^{4n}$. 
\end{restatable*} We will see in \cref{lem:equivmorphism} that an action of a Lie group $G$ on an integrable quaternionic manifold $M$ by quaternionic automorphisms $M$ is an action by $(\PGL_{n+1}(\H),\mathbb{P}(\H^{n+1}))$-morphisms on $M$. We use this to get the following necessary condition for the existence of an action of a Lie group $G$ by quaternionic automorphisms on a connected integrable quaternionic manifold in terms of representations of $\gog$. \begin{restatable*}{theorem}{necessarycondition}\label{thm:necessarycondition}
Let $G$ act non-trivially by quaternionic automorphisms on a connected integrable quaternionic manifold $M$ of dimension $4n$. Then there is a non-trivial representation $\rho:\gog \to \sl_{n+1}(\H)$.
\end{restatable*}

In \cref{sec:reps} we will describe all $\H$-representations of a real semisimple Lie algebra $\gog$, i.e., all Lie algebra homomorphisms $\rho:\gog\to \End_\H(W)$ for a finite dimensional $\H$-vector space $W$, following the work of Iwahori (\cite{Iwahori1959}) and Tits (\cite{Tits1967}). We will show that $\H$-representations are completely reducible (\cref{thm:completelyreducible}) and completely classify irreducible $\H$-representations (\cref{thm:irreduciblereps}). We can conclude that an $\H$-representation is uniquely defined by its underlying complex representation (\cref{thm:RHequivalentiffCequivalent}).

In \cref{sec:apps} we will use this knowledge to find the $\H$-representation of minimal dimension for any real simple Lie algebra $\gog$ (\cref{thm:listsminimalquatreps}). This provides us with a lower bound on the dimension of a connected manifold with $G$-invariant integrable almost quaternionic structure (\cref{cor:necessarycondition2}). For homogeneous spaces, the dimension of $G$ is an upper bound. We compare those bounds to prove our results. 

\section*{Acknowledgements}

This article is part of my PhD thesis at the Technion under the guidance of Michah Sageev and Tobias Hartnick. I want to thank Ilka Agricola and Bernhard M\"uhlherr for insights in the differential geometry and representation theory aspects, respectively. I thank the Mathematical Institute at Giessen University and the Department of Mathematics at Karlsruhe Institute of Technology (KIT) for excellent working conditions and financial support during my visits. I also want to acknowledge the generous financial support through the Department of Mathematics of the Technion.

\section{Quaternionic manifolds}\label{sec:preliminaries}

In this section we will explain the different definitions of quaternionic structures on manifolds and how they relate to each other. After giving the definitions in general, we will focus on the example most relevant to us, namely Riemannian symmetric spaces.

\subsection*{Types of quaternionic structures}\label{subsec:quaternionstructures}

There is an assortment of different structures on differential manifolds that are called quaternionic in the literature. We will mostly follow the nomenclature used in Chapter 14 of Besse (\cite{Besse2008}). Many of the following definitions are instances of more general definitions concerning $\cG$-structures on differentiable manifolds.

For every natural number $n\in \mathbb{N}$ we fix a standard $\R$-basis of the $n$-dimensional right quaternionic vector space $\H^n$, i.e., we fix an isomorphism $\H^n\cong\R^{4n}$. This also fixes an embedding $\GL(n,\H)\cdot\H^*\subset \GL(4n,\R)$.

\begin{definition}\label{def:almostquaternion}
An \emph{almost quaternionic structure} $\cB(M)$ on a smooth manifold $M^{4n}$ is a $\GL(n,\H)\cdot\H^*$-structure on $M$, i.e., $\cB(M)$ is a principal $\GL(n,\H)\cdot\H^*$-subbundle of the frame bundle $FM$ of $M$. An \emph{almost quaternionic manifold} is a smooth manifold $M$ with an almost quaternionic structure. A \emph{quaternionic morphism} between almost quaternionic manifolds $f:M^{4n}\to N^{4n}$ is a $\GL(n,\H)\cdot\H^*$-structure morphism, i.e., $f$ is a smooth map such that the induced map $f_*$ on the frame bundle satisfies $f_*(\cB(M))\subset\cB(N)$. We denote the group of quaternionic automorphisms of an almost quaternionic manifold $M$ by $\Aut_\H(M)$.
\end{definition}

This is the most basic definition and all other structures we mention are built upon it. For us a quaternionic morphism will always be a morphism of the underlying almost quaternionic structure.

Since $\GL(1,\H)\cdot\H^*=\R^+\times \operatorname{O}(4)$, an almost quaternionic structure on $M^4$ is a conformal structure. To avoid the different taste of this case, we will mostly restrict ourselves to manifolds of dimension $4n$ for $n\geq 2$.

\begin{remark}\label{rem:almostquaternion2}
One can show that a smooth manifold $M$ admits an almost quaternionic structure $\cB(M)$ if and only if it admits an algebra subbundle $\H(M)$ of the bundle of endomorphisms of the tangent bundle $TM$ with standard fiber $\H$. The correspondence is given as follows.

Let $\cB(M)$ be an almost quaternionic structure on $M$. A local section $\eta$ of $\cB(M)$ defined on a neighborhood $U\subset M$ is a smooth frame, i.e., a local trivialization of the tangent bundle over $U$. Using the fixed isomorphism between $\R^{4n}$ and $\H^n$, we get an $\R$-linear isomorphism ${\eta_x:T_xM\to \H^n}$ for any $x\in U$. We give a local trivialization of $\H(M)$ over $U$ by \begin{equation*}
\phi^{-1}(x,q)(v)=\eta_x^{-1}(\eta_x(v)\cdot q)
\end{equation*} for any $x\in U$, $q\in \H$ and $v\in T_xM$.

Conversely, let $\H(M)$ be a bundle as required above. There exists a local trivialization $\phi:\H(U)\to U\times \H$ on a neighborhood $U\subset M$ that respects the algebra structure. This makes $T_xM$ a quaternionic vector space for any $x\in U$. By shrinking $U$ if necessary we can choose vector fields $X_1,\dots,X_n$ such that $\{X_i(x)\}_{i=1}^n$ is an $\H$-basis of $T_xM$ for all $x\in U$. Again using the standard isomorphism between $\H^n$ and $\R^{4n}$, this choice of a basis gives us a local section $\eta$ of the frame bundle $FM$. We define $\cB(U)$ as the $\GL(n,\H)\cdot\H^*$-orbit of $\eta$.

A map $f:M \to N$ is a quaternionic morphism if and only if it is a local diffeomorphism $f:M \to N$ such that for every $x \in M$ and $v\in T_xM$ we have $df_x(v\cdot\H(M)_x)=df_x(v)\cdot\H(N)_{f(x)}.$
\end{remark}

\begin{example}\label{ex:quaternionvectorspace}
The quaternionic vector space $\H^n$ is an almost quaternionic manifold. Using the flat connection we identify the tangent bundle $T\H^n$ of $\H^n$ with $\H^n\times\H^n$. We set $\cB(\H^n)$ as the $\GL(n,\H)\cdot\H^*$-orbit of our standard $\R$-basis of $\H^n$. We get the same almost quaternionic structure on $\H^n$ by setting $\H(\H^n)=\{R_q\mid q\in \H\}$, where $R_q: T\H^n\to T\H^n,\ (p,v)\mapsto (p,vq)$ is the endomorphism of the tangent bundle given by right multiplication with $q$. 
\end{example}

The following definition does not play a role in our argument, but is prevalent in the literature. We mention it to avoid confusion.

\begin{definition}\label{def:quaternion}
A \emph{quaternionic manifold} is an almost quaternionic manifold admitting a torsion-free $\GL(n,\H)\cdot\H^*$-connection.
\end{definition}

The notion we want to investigate is the following:

\begin{definition}\label{def:integrablequaternion}
An almost quaternionic structure $\cB(M)$ on a smooth manifold $M^{4n}$ is \emph{integrable}, if for any $p\in M$ there is a coordinate neighborhood $p\in U$ and a coordinate chart $\phi=(x_1,\dots,x_{4n}):U\to \R^{4n}$ such that the coordinate vector fields $(\frac{\partial}{\partial x_1},\dots,\frac{\partial}{\partial x_{4n}})$ give a section of $\cB(M)$. We call an almost quaternionic manifold \emph{integrable}, if its almost quaternionic structure is integrable. An \emph{integrable quaternionic manifold} is a smooth manifold $M$ with an integrable almost quaternionic structure.
\end{definition}

Note for contrast that an almost complex manifold admitting a torsion-free $\GL(n,\C)$-connection is already complex. We will later see examples of quaternionic manifolds that are not integrable.

\begin{remark}
The almost quaternionic structure on $\H^n$ is integrable. Actually, more is true: An almost quaternionic structure on $M^{4n}$ is integrable if and only if it is locally isomorphic to the almost quaternionic structure on $\H^n$. Explicitly, $M^{4n}$ is an integrable quaternionic manifold if for every $x\in M$ there exists an open neighborhood $x\in U$ and a quaternionic isomorphism $\phi:U\to V\subset \H^n$ onto its image (see p. 315 in \cite{Sternberg1983}).
\end{remark}

In the context of Riemannian manifolds, we get yet another notion of quaternionic structure.

\begin{definition}\label{def:qKspace}
A $4n$-dimensional Riemannian manifold $M$ with $n\geq 2$ is \emph{quaternion-K\"ahler} if its holonomy group is contained in ${\Sp(n)\cdot\Sp(1)}$. A $4$-dimensional Riemannian manifold $M$ is \emph{quaternion-K\"ahler} if it is Einstein and self-dual.
\end{definition}

Often one demands in the definition of quaternion-K\"ahler manifolds that the holonomy group contains the $\Sp(1)$-factor of $\Sp(n)\cdot\Sp(1)$ to exclude the hyperk\"ahler case. We will not demand this. 

Since the Levi-Civita connection is torsion-free, a quaternion-K\"ahler manifold is a quaternionic manifold (see Proposition 14.63 in \cite{Besse2008}). So we have the diagram of implications:

\[\begin{tikzcd}
\text{almost quaternion} \arrow[r, Leftarrow] & \text{quaternion} \arrow[r, Leftarrow] \arrow[d, Leftarrow] & \text{integrable}\\
{} & \text{quaternion-K\"ahler} & {}
\end{tikzcd}\]

\subsection*{Quaternion-K\"ahler symmetric space}\label{subsec:symmetric spaces}

Our main interest are $G$-homogeneous spaces with $G$-invariant almost quaternionic structures, in particular Riemannian symmetric spaces. They will also provide examples for the definitions given above.

\begin{lemma}\label{lem:homogeneousspace}
Let $M^{4n}=G/H$ be a homogeneous space. Then $M$ has a $G$-invariant almost quaternionic structure if and only if the linear isotropy group $\Is_{eH}(H)$ is a subgroup of $\GL(n,\H)\cdot\H^*\subset \GL(T_{eH}M)$.
\end{lemma}

\begin{proof}
Let $M^{4n}=G/H$ have a $G$-invariant almost quaternionic structure $cB(M)$. Since $G$ acts by quaternionic isomorphisms, we get $h_*(\cB(M)_{eH})=cB(M)_{eH}$ for $h\in H$. We get that $\Is_{eH}(H)\subset \GL(n,\H)\cdot\H^*$, since $\cB(M)$ is a principal $\GL(n,\H)\cdot\H^*$-bundle by definition. 

For the other direction, let $M^{4n}=G/H$ be a homogeneous space such that the linear isotropy group $\Is_{eH}(H)$ is a subgroup of $\GL(n,\H)\cdot\H^*$. Then $\H(X)_{eH}=\H^*\cup\{0\}$ is $H$-invariant, since $\GL(n,\H)\cdot\H^*$ normalizes $H^*$. So $\H(X)_{gH}=dg_{eH}\circ \H(X)_{eH}\circ dg^{-1}_{gH}$ is well-defined. By \cref{rem:almostquaternion2} this gives $M$ an almost quaternionic structure. The quaternionic structure is $G$-invariant by definition.
\end{proof}

Now we can show that Riemannian symmetric spaces with invariant almost quaternionic structure are characterized by their holonomy.

\almostqsymmetricimpliesqK

\begin{proof}
Since $G/K$ is a Riemannian symmetric space, the isotropy representation of $K$ is faithful, i.e., $K=\Is_{eK}(K)$. The Riemannian holonomy of $G/K$ is $K$. If $G/K$ has a $G$-invariant almost quaternionic structure, then ${K=\Is_{eK}(K)\subset \GL(n,\H)\cdot\H^*}$ by \cref{lem:homogeneousspace}. The unique maximal compact subgroup of $\GL(n,\H)\cdot\H^*$ is $\Sp(n)\cdot\Sp(1)$. Since $K$ is compact, we get $K\subset \Sp(n)\cdot\Sp(1)$. So $G/K$ is quaternion-K\"ahler. If $G/K$ is quaternion-K\"ahler, we have ${\Is_{eK}(K)=K\subset \Sp(n)\cdot\Sp(1)}$. By \cref{lem:homogeneousspace} $G/K$ has a $G$-invariant almost quaternionic structure.
\end{proof}

\begin{definition}\label{def:qKsymmetricspace}
A \emph{quaternion-K\"ahler symmetric space} is a Riemannian symmetric space that is quaternion-K\"ahler.
\end{definition}

Wolf classified the quaternion-K\"ahler symmetric spaces in \cite{Wolf1965}. Every quaternion-K\"ahler symmetric space is either flat, of compact type or of non-compact type. In \cref{table:noncompqKsymspace} we list all quaternion-K\"ahler symmetric spaces of non-compact type. (Products of these spaces are not again quaternion-K\"ahler.) The compact quaternion-K\"ahler symmetric spaces are dual to those in \cref{table:noncompqKsymspace}. (Note that $\sp(1)=\su(2)$ and $\so(4)=\su(2)\oplus\su(2)$.)

\begin{example}\label{ex:quaternionvectorspacesymmetric}
A flat quaternion-K\"ahler symmetric space is isometric to $\H^n$. In particular, it is an integrable quaternionic manifold as we saw in \cref{ex:quaternionvectorspace}.
\end{example}

\begin{example}\label{ex:quaternionprojectivespace}
The quaternionic projective space $\mathbb{P}(\H^{n+1})$ is the quotient space for the action of $\H^*$ on $\H^{n+1}$ by right multiplication. Let ${U_i=\{[q_1,\dots,q_{n+1}]\mid q_i\neq 0\}}$ and \begin{equation*}
\phi_i:U_i\to \H^n\cong \R^{4n},\quad [q_1,\dots,q_{n+1}]\mapsto (q_1,\dots,q_{i-1},\hat{q_i},q_{i+1},\dots,q_{n+1})q_i^{-1}.
\end{equation*}Then the atlas $\{U_i,\phi_i\}$ makes $\mathbb{P}(\H^{n+1})$ an integrable quaternionic manifold (see Lemma 3.2. in \cite{Kulkarni1978}). $\PGL_{n+1}(\H)$ acts on $\mathbb{P}(\H^{n+1})$ by quaternionic morphisms.
\end{example}

Actually more is true: By Theorem 5.1. in \cite{Kulkarni1978} we have $\PGL_{n+1}(\H)=\Aut_{\H}(\mathbb{P}(\H^{n+1}))$. This fact will play a major role in \cref{sec:Kulkarni}. For the next example, see also Preliminaries in \cite{LiuZhang2012}.

\begin{example}\label{ex:quaternionhyperbolicspace}
The $n$-dimensional quaternionic ball $\mathbb{B}_\H^n=\{q \in \H^n\mid |q| < 1\}$ can be embedded into $\H\mathbb{P}^n$ by \begin{equation*}
\mathbb{B}_\H^n\to\H\mathbb{P}^n, q=(q_1,\dots,q_n)^T\mapsto [q_1,\dots,q_n,1].
\end{equation*} This map is a quaternionic isomorphism, so we can identify $\mathbb{B}_\H^n$ with its image in $\H\mathbb{P}^n$. $\Sp(n,1)$ acts transitively on $\mathbb{B}_\H^n$ and $\Sp(n)\times\Sp(1)$ is the stabilizer of the point $[0,\dots,0,1]$. So the quaternionic hyperbolic space $H^n_\H=\Sp(n,1)/\Sp(n)\times\Sp(1)$ can be realized as $\mathbb{B}_\H^n$. In particular, $H^n_\H$ is an integrable quaternionic manifold.
\end{example}

\begin{table}[!htbp]\begin{adjustbox}{center}
\begin{tabular}{ |c|c|c|c|c|c| } \hline
$\gog_\C$ & $\gog$ & $\gok$ & Satake-Tits diagram & Type of $\Sigma$ & $\dim(G/K)$\\
\hline
$A_2$ & $\su(1,2)$ & $\su(1)\oplus\su(2)$ & \adjustbox{valign=m}{\begin{dynkinDiagram}[labels*={1,2}, edge length=.75cm]{A}{oo}\invol{1}{2}\end{dynkinDiagram}} & $BC_1$ & $4$\\
$A_3$ & $\su(2,2)$ & $\su(2)\oplus\su(2)$ & \adjustbox{valign=m}{\begin{dynkinDiagram}[labels*={1,2,3},edge length=.75cm]{A}{ooo}\invol{1}{3}\end{dynkinDiagram}} & $B_2$ & $8$\\
$A_4$ & $\su(3,2)$ & $\su(3)\oplus\su(2)$ & \adjustbox{valign=m}{\begin{dynkinDiagram}[labels*={1,2,3,4},edge length=.75cm]{A}{oooo}\invol{1}{4}\invol{2}{3}\end{dynkinDiagram}} & $BC_2$ & $12$\\
$A_n$, $n>4$ & $\su(n-1,2)$ & $\su(n-1)\oplus\su(2)$ & \adjustbox{valign=m}{\begin{dynkinDiagram}[labels*={1,2,3,n-2,n-1,n},edge length=.75cm]{A}{oo*.*oo}\invol{1}{6}\invol{2}{5}\end{dynkinDiagram}} & $BC_2$ & $4(n-1)$\\
$B_2$ & $\so(1,4)$ & $\so(1)\oplus\so(4)$ & \adjustbox{valign=m}{\begin{dynkinDiagram}[labels*={1,2},edge length=.75cm]{B}{o*}\end{dynkinDiagram}} & $A_1$ & $4$\\
$B_3$ & $\so(3,4)$ & $\so(3)\oplus\so(4)$ & \adjustbox{valign=m}{\begin{dynkinDiagram}[labels*={1,2,3},edge length=.75cm]{B}{ooo}\end{dynkinDiagram}} & $B_3$ & $12$\\
$B_4$ & $\so(5,4)$ & $\so(5)\oplus\so(4)$ & \adjustbox{valign=m}{\begin{dynkinDiagram}[labels*={1,2,3,4},edge length=.75cm]{B}{oooo}\end{dynkinDiagram}} & $B_4$ & $20$\\
$B_n$, $n>4$ & $\so(2n-3,4)$ & $\so(2n-3)\oplus\so(4)$ & \adjustbox{valign=m}{\begin{dynkinDiagram}[labels*={1,2,3,4,5,n-1,n},edge length=.75cm]{B}{oooo*.**}\end{dynkinDiagram}} & $B_4$ & $4(2n-3)$\\
$C_3$ & $\sp(2,1)$ & $\sp(2)\oplus\sp(1)$ & \adjustbox{valign=m}{\begin{dynkinDiagram}[labels*={1,2,3},edge length=.75cm]{C}{*o*}\end{dynkinDiagram}} & $BC_1$ & $8$\\
$C_n$, $n>3$ & $\sp(n-1,1)$ & $\sp(n-1)\oplus\sp(1)$ & \adjustbox{valign=m}{\begin{dynkinDiagram}[labels*={1,2,3,n-1,n},edge length=.75cm]{C}{*o*.**}\end{dynkinDiagram}} & $BC_1$ & $4(n-1)$\\
$D_4$ & $\so(4,4)$ & $\so(4)\oplus\so(4)$ & \adjustbox{valign=m}{\begin{dynkinDiagram}[labels={1,2,3,4},edge length=.75cm]{D}{oooo}\end{dynkinDiagram}} & $D_4$ & $16$\\
$D_5$ & $\so(6,4)$ & $\so(6)\oplus\so(4)$ & \adjustbox{valign=m}{\begin{dynkinDiagram}[labels={1,2,3,4,5},edge length=.75cm]{D}{ooooo}\invold{4}{5}\end{dynkinDiagram}} & $B_4$ & $24$\\
$D_6$ & $\so(8,4)$ & $\so(8)\oplus\so(4)$ & \adjustbox{valign=m}{\begin{dynkinDiagram}[labels={1,2,3,4,5,6},edge length=.75cm]{D}{oooo**}\end{dynkinDiagram}} & $B_4$ & $32$\\
$D_n$, $n>6$ & $\so(2n-4,4)$ & $\so(2n-4)\oplus\so(4)$ & \adjustbox{valign=m}{\begin{dynkinDiagram}[labels={1,2,3,4,5,n-2,n-1,n}, edge length=.75cm]{D}{oooo*.***}\end{dynkinDiagram}} & $B_4$ & $4(2n-4)$\\
$E_6$ & $\goe_{6(2)}$ & $\su(6)\oplus\su(2)$ & \adjustbox{valign=m}{\begin{dynkinDiagram}[labels*={1,6,2,3,4,5}, edge length=.75cm]{E}{oooooo}\invol{1}{6}\invol{3}{5}\end{dynkinDiagram}} & $F_4$ & $40$\\
$E_7$ & $\goe_{7(-5)}$ & $\so(12)\oplus\su(2)$ & \adjustbox{valign=m}{\begin{dynkinDiagram}[labels*={6,7,5,4,3,2,1}, backwards, edge length=.75cm]{E}{o*oo*o*}\end{dynkinDiagram}} & $F_4$ & $64$\\
$E_8$ & $\goe_{8(-24)}$ & $\goe_7\oplus\su(2)$ & \adjustbox{valign=m}{\begin{dynkinDiagram}[labels*={7,8,6,5,4,3,2,1}, backwards, edge length=.75cm]{E}{o****ooo}\end{dynkinDiagram}} & $F_4$ & $112$\\
$F_4$ & $\gof_{4(4)}$ & $\sp(3)\oplus\su(2)$ & \adjustbox{valign=m}{\begin{dynkinDiagram}[labels*={1,2,3,4},mark=o, edge length=.75cm, reverse arrows]{F}{4}\end{dynkinDiagram}} & $F_4$ & $28$\\
$G_2$ & $\gog_{2(2)}$ & $\su(2)\oplus\su(2)$ & \adjustbox{valign=m}{\begin{dynkinDiagram}[labels*={1,2},mark=o, edge length=.75cm, reverse arrows]{G}{2}\end{dynkinDiagram}} & $G_2$ & $8$\\
\hline
\end{tabular}\end{adjustbox}
\caption{Quaternion-K\"ahler symmetric spaces of non-compact type}
\label{table:noncompqKsymspace}\end{table}

We saw above that quaternionic vector space, projective spaces and hyperbolic space are integrable quaternion-K\"ahler symmetric spaces. We will show below that no other quaternion-K\"ahler symmetric space is integrable.

\section{From integrable quaternionic manifolds to quaternionic representations}\label{sec:Kulkarni}

In this section we want to pass from integrable quaternionic manifolds to manifolds with a $(\PGL_{n+1}(\H),\mathbb{P}(\H^{n+1}))$-structure in the sense of Thurston. \cref{thm:kulkarniquaternion} provides us with a bridge between these notions. This different approach allows us to pass from actions of a Lie group $G$ on integrable quaternionic manifolds by quaternionic automorphisms to particular representations of $\gog$.

\subsection*{\texorpdfstring{$(G,X)$-structures}{(G,X)-structures}}\label{subsec:GXstructures}

We first need to review some facts about general $(G,X)$-structures. For this we will follow \cite{Thurston1997}.

\begin{definition}\label{def:(G,X)-structures}
Let $G$ act on $X$ analytically, by diffeomorphisms. A \emph{differentiable $(G,X)$-atlas} on a differentiable manifold $M$ is an atlas $\{(U_i, \phi_i)\}_{i\in I}$ such that\begin{enumerate}
\item the $U_i$ form an open cover of $M$,
\item the \emph{charts} $\phi_i:U_i\to X$ are smooth embeddings such that \begin{equation*}
\phi_i\circ\phi_j^{-1}:\phi_j (U_i\cap U_j)\to \phi_i (U_i\cap U_j)
\end{equation*}is a restriction of an element $g_{ij}\in G$.
\end{enumerate} Two $(G,X)$-atlases on $M$ are \emph{equivalent} if their union is a $(G,X)$-atlas. A \emph{$(G,X)$-structure} on $M$ is an equivalence class of differentiable $(G,X)$-atlases. A \emph{$(G,X)$-manifold} is a differentiable manifold with a $(G,X)$-structure. A \emph{$(G,X)$-morphism} is a local diffeomorphism ${f:M\to N}$ between $(G,X)$-manifolds $M$ and $N$ such that for any pair of charts $(U, \phi)$ of $M$ and $(V, \psi)$ of $N$ we have that \begin{equation*}\psi\circ f\circ \phi^{-1}:\phi(U\cap f^{-1}(V))\to \psi(f(U)\cap V)\end{equation*}is a restriction of an element $g\in G$. We denote the group of $(G,X)$-automorphisms of a $(G,X)$-manifold $M$ by $\Aut_{(G,X)}(M)$.
\end{definition}

\begin{definition}\label{def:analyticcontinuation}
Let $M$ be a $(G,X)$-manifold, $\alpha:[0,1]\to M$ a path in $M$ and $(U_0, \phi_0)$ a chart such that $\alpha(0)\in U_0$. We subdivide $[0,1]$ into intervals $[x_i,x_{i+1}]$ with $x_0=0$, $x_{m_\alpha}=1$ such that $\alpha([x_i,x_{i+1}])$ is contained in a chart $(U_i,\phi_i)$ for $0\leq i<m_\alpha$. We denote by $\phi([\alpha],\phi_0)$ the \emph{analytic continuation} of $\phi_0$ along $\alpha$, that is \begin{equation*}
\phi([\alpha],\phi_0)=g_{01}\dots g_{m_\alpha-1,m_\alpha}\phi_{m_\alpha}.
\end{equation*}
\end{definition}

The notation $\phi([\alpha],\phi_0)$ is justified by the fact that the analytic continuation $\phi([\alpha],\phi_0)$ only depends on the homotopy class $[\alpha]$ (with fixed endpoints) and the chart $\phi_0$.

\begin{definition}\label{def:developingmap}
Let $M$ be a $(G,X)$-manifold. We choose a basepoint $x_0\in M$ and a chart $(U_0, \phi_0)$ such that $x_0\in U_0$. Let $\pi:\tilde{M}\to M$ be the universal cover of $M$. We identify $\tilde{M}$ with the space of homotopy classes $[\alpha]$ of curves $\alpha:[0,1]\to M$ starting at $x_0=\alpha(0)$. The \emph{developing map} is the map \begin{equation*}
D:\tilde{M}\to X,\quad [\alpha]\mapsto \phi([\alpha],\phi_0)(\alpha(1)).
\end{equation*}
\end{definition}

We can induce a $(G,X)$-structure on the universal cover $\tilde{M}$ of a $(G,X)$-manifold $M$ such that the covering map is a $(G,X)$-morphism. With respect to this structure the developing map $D$ is a $(G,X)$-morphism (see p.140 of \cite{Thurston1997}). For any $x\in \tilde{M}$ there is a neighborhood $U_x$ on which $D$ is a diffeomorphism on its image. So for any $f\in \Aut_{(G,X)}(\tilde{M})$ we have that \begin{equation*}
D\circ f \circ (D_{\restriction U_x})^{-1}:D(U_x)\to D(f(U_x))
\end{equation*} is the restriction of a unique element of $G$, which we denote by $\hol(f,x)$. The map $\hol(f,-):\tilde{M}\to G, x\mapsto \hol(f,x)$ is locally constant by definition. Since $\tilde{M}$ is simply connected, $\hol(f,-)$ is constant and we denote its image by $\hol(f)$. (This is an extension of the holonomy defined on p.151 of \cite{Thurston1997}.) The developing map $D$ is $\hol$-equivariant, i.e., \begin{equation*}
D\circ f=\hol(f)D.
\end{equation*} So $\hol:\Aut_{(G,X)}(\tilde{M})\to G$ is a homomorphism.

\subsection*{Kulkarni's theorem and consequences}\label{subsec:Kulkarni}

We are now prepared to cross the bridge from integrable quaternionic manifolds to quaternionic representations. The following definition can be made for any $\cG$-structure, but we stick to our almost quaternionic structures.

\begin{definition}\label{def:GGXcompatible}
Let $M$ be a $(G,X)$-manifold and let $X$ be an almost quaternionic manifold. An almost quaternionic structure on $M$ is \emph{compatible with the $(G,X)$-structure} if the charts are quaternionic morphisms.
\end{definition}

We can finally reformulate Kulkarni's theorem:

\kulkarni

The proof of the theorem relies on two independently important steps:
\begin{enumthm}
\item\label{thm:projspaceintegrable} $\mathbb{P}(\H^{n+1})$ has an integrable almost quaternionic structure (as seen in \cref{ex:quaternionprojectivespace}).
\item\label{thm:extendquaterisomorphism} Every locally defined quaternionic isomorphism of $\mathbb{P}(\H^{n+1})$ is a restriction of an element of $\PGL_{n+1}(\H)$.
\end{enumthm}

Actually Kulkarni proved a more general result: For $\cG\subset \GL_n(\R)$ of finite type he proved that there is a manifold $S(\cG)$ and a Lie group $K(\cG)$ acting on $S(\cG)$ such that a $\cG$-structure on a smooth manifold $M^{n}$ is integrable if and only if it is compatible with a $(K(\cG),S(\cG))$-structure on $M^{n}$ (Theorem 2.3 in \cite{Kulkarni1978}). He then proved that $S(\GL(n,\H)\cdot\H^*)=\mathbb{P}(\H^{n+1})$ and $K(\GL(n,\H)\cdot\H^*)=\PGL_{n+1}(\H)$ (Theorem 3.4 in \cite{Kulkarni1978}).

A clarifying and useful observation for us was the following:

\begin{lemma}\label{lem:equivmorphism}
Let $M^{4n}$ be an integrable quaternionic manifold. Then \begin{equation*}
\Aut_\H(M)=\Aut_{(\PGL_{n+1}(\H),\mathbb{P}(\H^{n+1}))}(M).
\end{equation*}
\end{lemma}

\begin{proof}
By \cref{thm:kulkarniquaternion} we know that $M^{4n}$ has an almost quaternionic structure compatible with its $(\PGL_{n+1}(\H),\mathbb{P}(\H^{n+1}))$-structure. Let ${f\in \Aut_{(\PGL_{n+1}(\H),\mathbb{P}(\H^{n+1}))}(M)}$. For any pair of charts $(U_\alpha, \phi_\alpha)$, $(U_\beta, \phi_\beta)$ in the $(\PGL_{n+1}(\H),\mathbb{P}(\H^{n+1}))$-atlas on $M$, we get that $\phi_\beta\circ f\circ \phi_\alpha^{-1}$ is a restriction of an element in $\PGL_{n+1}(\H)$. By \cref{thm:extendquaterisomorphism} we know $\PGL_{n+1}(\H)=\Aut_\H(\mathbb{P}(\H^{n+1}))$. So $\phi_\beta\circ f\circ \phi_\alpha^{-1}$ is a quaternionic morphism. Since the charts are quaternionic morphisms, we get $f\in \Aut_\H(M)$.

Let $f\in \Aut_\H(M)$. Then for any pair of charts $(U_\alpha, \phi_\alpha)$, $(V_\beta, \psi_\beta)$ of $M$ we have that \begin{equation*}
\psi_\beta\circ f\circ \phi_\alpha^{-1}:\phi_\alpha(U_\alpha\cap f^{-1}(V_\beta))\to \psi_\beta(f(U_\alpha)\cap V_\beta)
\end{equation*} is a locally defined quaternionic isomorphism of $\mathbb{P}(\H^{n+1})$ and therefore by \cref{thm:extendquaterisomorphism} a restriction of an element of $\PGL_{n+1}(\H)$. So ${f\in \Aut_{(\PGL_{n+1}(\H),\mathbb{P}(\H^{n+1}))}(M)}$.
\end{proof}

We get a first application of this approach to integrable quaternionic manifolds (see p. 412 in \cite{Besse2008}):

\begin{corollary}\label{cor:compactcase}
A compact simply-connected manifold $M^{4n}$ has an integrable almost quaternionic structure if and only if $M$ is diffeomorphic to $\mathbb{P}(\H^{n+1})$. In particular, a compact quaternion-K\"ahler symmetric space $X$ of dimension $4n$ is integrable if and only if $X=\mathbb{P}(\H^{n+1})$.
\end{corollary}

\begin{proof}
Assume $M$ has an integrable almost quaternionic structure. Then $M$ is a $(\PGL_{n+1}(\H),\mathbb{P}(\H^{n+1}))$-manifold by \cref{thm:kulkarniquaternion}. $M$ is simply connected, so the developing map gives us a local diffeomorphism ${D:M^{4n}\to\mathbb{P}(\H^{n+1})}$. Since $M$ is compact, $D$ is proper. Since $D$ is a proper local diffeomorphism, it is a covering map (see Proposition 4.46 in \cite{Lee2013}). Since $\mathbb{P}(\H^{n+1})$ is simply connected, the developing map actually has to be a diffeomorphism $M\to \mathbb{P}(\H^{n+1})$. 

Any compact quaternion-K\"ahler symmetric space is simply connected. So if $X$ is integrable, it is diffeomorphic to $\mathbb{P}(\H^{n+1})$. But no other simply connected compact Riemannian symmetric space is even homeomorphic to $\mathbb{P}(\H^{n+1})$ (see Appendix A of \cite{Terzic2003}).
\end{proof}

So the compact quaternion-K\"ahler symmetric spaces other than $\mathbb{P}(\H^{n+1})$ are examples of quaternionic manifolds that are not integrable.

We now use \cref{thm:kulkarniquaternion} to find a necessary condition for the existence of an action of a Lie group $G$ by quaternionic automorphisms on an integrable quaternionic manifold in terms of representations of $\gog$. We will check this condition in the later sections.

\necessarycondition

\begin{proof}
Assume first that $M$ is simply connected. Since $\cB(M)$ is integrable, it is compatible with a $(\PGL_{n+1}(\H),\mathbb{P}(\H^{n+1}))$-structure on $M$ by \cref{thm:kulkarniquaternion}. Since $M$ is simply connected, we have $\hol:\Aut_{(\PGL_{n+1}(\H),\mathbb{P}(\H^{n+1}))}(M)\to \PGL_{n+1}(\H)$. Since $\cB(M)$ is $G$-invariant, we can restrict to $\hol:G\to \PGL_{n+1}(\H)$ by \cref{lem:equivmorphism}. The representation $\hol:G\to \PGL_{n+1}(\H)$ is not trivial, since $G$ acts non-trivially on $M$ and the preimage of a point $D^{-1}(x)$ is discrete for $x\in \mathbb{P}(\H^{n+1})$. We have $\PSL_{n+1}(\H)=\PGL_{n+1}(\H)$ and $\SL_{n+1}(\H)$ covers $\PSL_{n+1}(\H)$. Taking the derivative we get $\rho:\gog \to \sl_{n+1}(\H)$.

If $M$ is not simply connected, there is a covering group $G'$ of $G$ that acts on $\tilde{M}$ covering the action of $G$ on $M$ (see Chapter I, Section 9 in \cite{Bredon1972} or Chapter 1 in \cite{Onishchik1994} for connected $G$). So by \cref{thm:kulkarniquaternion} $\tilde{M}$ has a $G'$-invariant integrable almost quaternionic structure. Since $G$ and $G'$ are locally isomorphic, we can use the first case.
\end{proof}

One can prove this theorem for general $\cG$-structures and get a non-trivial representation $\rho:\gog \to \mathfrak{k}(\cG)$. But this is only useful if one can compute $K(\cG)$ and there is a theory of representations into $\mathfrak{k}(\cG)$. We do not know any $\cG\neq \GL(n,\H)\cdot\H^*$ for which this is the case.

Studying $\hol$ and the $\hol$-equivariant developing map $D$ closer, one can get collect more information about the representation in \cref{thm:necessarycondition}. This might be an interesting approach if one aims for a classification of $G$-homogeneous spaces with $G$-invariant almost quaternionic structure (at least up to dimension). But we will not attempt this here. 

\section{Quaternionic representations}\label{sec:reps}

\cref{thm:necessarycondition} motivates us to study non-trivial representations ${\rho:\gog \to \sl_{n+1}(\H)}$. We focus first on the quaternionic aspect of these representations. Most of the material in this section appears in some form in \cite{Iwahori1959} and \cite{Tits1967}, see also Chapter 8 of \cite{Onishchik2004}.

\begin{definition}\label{def:rchrep}
Let $\gog$ be a real Lie algebra. Let $\bK\in\{\R,\C,\H\}$. An ($\R$-linear) Lie algebra homomorphism $\rho:\gog\to \End_\bK(W)$ for a finite dimensional $\bK$-vector space $W$ is a \emph{$\bK$-representation}. An \emph{isomorphism of $\bK$-representations} is an equivariant $\bK$-linear isomorphism. A $\bK$-representation is \emph{irreducible} if it has no proper $\bK$-subrepresentations.
\end{definition}

Our goal is to find all $\H$-representations of a real semisimple Lie algebra $\gog$. We will proceed in the following steps:\begin{enumerate}
\item Characterize $\H$-representations by $\C$-representations with antiinvolutions.
\item Use this to decompose $\H$-representations as direct sums of two types of $\H$-representations. Check those two types. Conclude complete reducibility of $\H$-representations. Characterize irreducible $\H$-representations in terms of \emph{quaternionic} irreducible $\C$-representations.
\item Reduce the question to a question about fundamental representations of simple complex Lie algebras and finally answer it.
\end{enumerate} We will do the classification of irreducible $\R$-representations of $\gog$ at the same time, since the proof runs in parallel. 

\subsection*{\texorpdfstring{$\H$-representations}{H-representations} as \texorpdfstring{$\C$-representations}{C-representations} with antiinvolutions}

As a first step, we want to phrase our question in terms of representations on complex vector spaces.

\begin{definition}\label{def:antiinvolutions}
Let $V$ be a complex vector space. An antilinear map $J:V\to V$ such that $J^2=1$ (resp. $J^2=-1$) is called \emph{antiinvolution of the first (resp. second) kind}. If $J_V, J_W$ are both antiinvolution of the first (resp. second) kind, we call a $\C$-linear map $f:V\to W$ that satisfies $J_W\circ f=f\circ J_V$ a \emph{morphism of the first (resp. second) kind}. We denote the set of morphisms of the first (resp. second) kind by $\Hom_\C(J_V,J_W)$.
\end{definition}

\begin{remark}\label{rem:equivalentcategories}
The category of quaternionic vector spaces is equivalent to the category of complex vector spaces with antiinvolutions of the second kind. A quaternionic vector space $W$ corresponds to a complex vector space $W_\C$ by restriction of scalars from $\H$ to $\C$. The multiplication with $j\in \H$ on $W$ is an antiinvolution of the second kind on $W_\C$. Given an $\H$-linear map $f:W_1\to W_2$, we get a morphism of the second kind $f\in\Hom_\C(J_1,J_2)$. In the other direction, given a complex vector spaces $V$ with an antiinvolution of the second kind $J$, we get a quaternionic vector space by equipping $V$ with a right $\H$-module structure such that multiplication with $j\in \H$ is given by $J$. A morphism of the second kind $f\in\Hom_\C(J_1,J_2)$ is in particular a $\C$-linear map.

The category of real vector spaces is equivalent to the category of complex vector spaces with antiinvolutions of the first kind. A real vector space $W$ corresponds to a complex vector space $W\otimes_\R \C$. The complex conjugation on $W\otimes_\R \C$ is an antiinvolution of the first kind. Given an $\R$-linear map $f:W_1\to W_2$, we get the morphism of the first kind $f\otimes_\R 1\in\Hom_\C(J_1,J_2)$. In the other direction, given a complex vector spaces $V$ with an antiinvolution of the first kind $J$, we get a real vector space $\Fix(J)$. Given a morphism of the first kind $f\in\Hom_\C(J_1,J_2)$, we get an $\R$-linear map $f_{\restriction\Fix(J_1)}:\Fix(J_1)\to \Fix(J_2)$.
\end{remark}

These equivalences enable us to see $\R$-representations and $\H$-representations as $\C$-representations with antiinvolutions:\begin{enumerate}
\item An $\R$-representation $\rho:\gog\to \End_\R(W)$ corresponds to a representation $\rho:\gog\to \End_\C(W\otimes_\R \C,J)$ into endomorphisms of the first kind. 
\item An $\H$-representation $\rho:\gog\to \End_\H(W)$ corresponds to a representation $\rho:\gog\to \End_\C(W_\C,J)$ into endomorphisms of the second kind.
\end{enumerate} From now on, we regard every $\R$- or $\H$-representation as a pair $(\rho,J)$ of a $\C$-representation $\rho$ and an antiinvolution $J$ of the relevant kind.

\begin{definition}
We say that a $\C$-representation $\rho$ \emph{admits an antiinvolution $J$} or that \emph{$J$ is an antiinvoluton on $\rho$}, if $(\rho,J)$ is an $\R$- or $\H$-representation. We call a $\C$-representation $\rho$\begin{enumerate}
\item \emph{real} if it admits an antiinvolution of the first kind.
\item \emph{quaternionic} if it admits an antiinvolution of the second kind.
\end{enumerate}
\end{definition}

Note that a $\C$-representation $\rho$ can be both real and quaternionic. We will see examples of this later (see \cref{lem:conjugateantiinvolutions}). Note also that a quaternionic (resp. real) representation $\rho$ could a priori admit two antiinvolution of the second (resp. first) kind $J_1$, $J_2$ such that $(\rho,J_1)$ and $(\rho, J_2)$ are non-equivalent $\H$-representations (resp. $\R$-representations). We will see in \cref{thm:RHequivalentiffCequivalent} that this is not the case.

In the literature, a $\C$-representation that does not admits an antiinvolution is sometimes called \emph{complex} (for example in \cite{Onishchik2004}). We will not use this term to avoid confusion. Instead we will characterize these representations in terms of conjugation.

\begin{definition}\label{def:conjugatevectorspace}
Let $V$ be a complex vector space. The \emph{complex conjugate vector space} $\overline{V}$ has the same additive group structure as $V$, but is equipped with the scalar multiplication $z*v=\overline{z}v$. A map $f:V\to W$ gives rise to a map $\overline{f}:\overline{V}\to \overline{W},v\mapsto f(v)$.
\end{definition}

Note that $\overline{\overline{V}}=V$. If $f$ is $\C$-linear or antilinear, then so is $\overline{f}$.

\begin{definition}\label{def:conjugaterepresentation}
The \emph{conjugate} of a $\C$-representation $\rho:\gog\to \End_\C(V)$ is the representation $\overline{\rho}:\gog\to\End_\C(\overline{V})$ given by $\overline{\rho}(X)=\overline{\rho(X)}$ for any $X\in \gog$. A representation $\rho:\gog\to \End_\C(V)$ is \emph{self-conjugated} if it is equivalent to its conjugate, i.e., if there is an equivariant $\C$-linear isomorphism $f:\overline{V}\to V$.
\end{definition}

Equivalently, a $\C$-representation $\rho:\gog\to \End_\C(V)$ is self-conjugated if there is an equivariant antilinear bijective map and therefore an equivariant antiinvolution $J:V\to V$. So a $\C$-representation is not self-conjugated if and only if it does not admit an antiinvolution.

We state the following easy observation, since we will use it a lot.

\begin{lemma}\label{lem:ubiquitousobservation}
Let $J:V\to V$ be an antilinear, $\rho$-equivariant map. Then every other antilinear, $\rho$-equivariant map $J':V\to V$ is given by $J'=f\circ J$ for some $f\in \End_\C(\rho)$.
\end{lemma}

\begin{proof}
$J'J^{-1}$ is a $\C$-linear, $\rho$-equivariant map and therefore an element of $\End_\C(\rho)$.
\end{proof}

\subsection*{From \texorpdfstring{$\H$-representations}{H-representations} to quaternionic representations}

By Weyl's theorem on complete reducibility every $\R$- or $\C$-representation of $\gog$ decomposes as a direct sum of irreducible $\R$- or $\C$-representations, respectively. This is also true for $\H$-representations. As a step toward proving this, we decompose $\H$-representations as direct sums of two types of $\H$-representations.

\begin{lemma}\label{lem:reduciblecandidates}
Every $\H$-representation $(\rho,J)$ of a real semisimple Lie algebra $\gog$ decomposes as a direct sum of $\H$-representations $(\rho_i,J_i)$, which are of one of the following two types:\begin{enumerate}
\item $\rho_i$ is an irreducible quaternionic representation.
\item $\rho_i=\phi\oplus \overline{\phi}$ for an irreducible $\C$-representation $\phi$.
\end{enumerate}
\end{lemma}

\begin{proof}
Let $(\rho,J)$ be a $\H$-representation. Then $\rho$ is a $\C$-representation, which corresponds to a $\C$-linear representation $\rho_\C$ of the complex semisimple Lie algebra $\gog_\C$. By Weyl's theorem on complete irreducibility (see p. 241 in \cite{Knapp2002}), we know that $\rho_\C$ decomposes as a direct sum of irreducible representations of $\gog_\C$. Restricting again to $\gog$, we get a decomposition of $\rho$ into a direct sum of irreducible $\C$-representations $\phi_j:\gog\to \End_\C(V_j)$ for $1\leq j\leq k$. Since $J$ is $\rho$-equivariant and antilinear, $JV_1$ is a $\rho$-invariant complex subspace of $V$ and so $JV_1=V_j$ for some $1\leq j\leq k$. If $JV_1=V_1$, then we can set $(\rho_1,J_1)=(\phi_1,J_{\restriction V_1})$. So $\rho_1$ is an irreducible quaternionic representation and the statement follows by induction. If $JV_1\neq V_1$, we can assume without loss of generality that $JV_1=V_2$. Then by definition $\phi_2$ is conjugated to $\phi_1$. We set $(\rho_1,J_1)=(\phi_1\oplus\phi_2, J_{\restriction V_1\oplus V_2})$. Then $\rho_1\sim \phi_1\oplus \overline{\phi_1}$ and the statement follows by induction.
\end{proof}

This motivates us to study these types in detail.

\begin{lemma}\label{lem:oneequivalent}
Let $\rho$ be an irreducible $\C$-representation. Then it either does not admit an antiinvolution of any kind or a unique one up to equivalence. 
\end{lemma}

\begin{proof}
Let $J_1, J_2$ be two antiinvolutions on $\rho$. Using \cref{lem:ubiquitousobservation} and Schur's Lemma we get $J_2=zJ_1$ for $z\in \C\setminus \{0\}$. Since $J_2^2=zJ_1zJ_1=\lvert z\rvert^2 J_1^2$ we get that $J_1, J_2$ are of the same kind and $z=e^{i\theta}$ for $\theta\in [0,2\pi)$. Let $x=e^{\frac{1}{2}i\theta}$. Then $x\Id\in\End_\C(\rho)$ and $J_2x=z\overline{x}J_1=xJ_1$, so $(\rho,J_1)$ and $(\rho,J_2)$ are equivalent.
\end{proof}

\begin{definition}\label{def:oneirreducible}
Let $\rho$ be an irreducible $\C$-representation. If $\rho$ admits an antiinvolution, we denote it by $J_\rho$. We also define the following abbreviaton: \begin{equation*}
\varepsilon(\rho)=\begin{cases}
-1, &\text{ if $\rho$ is quaternionic}\\
\varnothing, &\text{ if $\rho$ is not self-conjugated}\\
1, &\text{ if $\rho$ is real}\end{cases}.
\end{equation*}
\end{definition}

\begin{corollary}\label{cor:oneirreducible}
Let $\rho$ be an irreducible $\C$-representation, that admits an antiinvolution $J$. Then the $\R$- or $\H$-representation $(\rho,J_\rho)$ is also irreducible.
\end{corollary}

\begin{proof}
Let $(\rho',J')$ be a $\R$- or $\H$-subrepresentation of $(\rho,J)$. Then in particular, $\rho'$ is a $\C$-subrepresentation of $\rho$. Since $\rho$ is irreducible, we get $\rho'=\rho$ and therefore $(\rho',J')=(\rho,J)$.
\end{proof}

The converse is not true: Not every irreducible $\H$-representation (or $\R$-representation) $(\rho,J)$ comes from an irreducible $\C$-representation $\rho$. We will give examples of this soon.

\begin{lemma}\label{lem:conjugateantiinvolutions}
Let $\rho:\gog\to \End_\C(V)$ be a $\C$-representation of $\gog$.\begin{enumlem}
\item\label{lem:conjugatereal} Set \begin{equation*} J_\R:V\oplus\overline{V}\to V\oplus\overline{V},\ (v_1,v_2)\mapsto (v_2,v_1).\end{equation*}$J_\R$ is an antiinvolution of the first kind on $\rho\oplus\overline{\rho}$. The $\R$-representation $(\rho\oplus\overline{\rho},J_\R)$ is equivalent to the $\R$-representation ${\phi:\gog\to\End_\R(V)}$ given by $\phi(X)=\rho(X)$ for all $X\in\gog$.
\item\label{lem:conjugatequaternionic} Set \begin{equation*}
J_\H:V\oplus\overline{V}\to V\oplus\overline{V},\ (v_1,v_2)\mapsto (-v_2,v_1).
\end{equation*}$J_\H$ is an antiinvolution of the second kind on $\rho\oplus\overline{\rho}$. The $\H$-representation $(\rho\oplus\overline{\rho},J_\H)$ is equivalent to the $\H$-representation $\phi:\gog\to\End_\H(V\otimes_\C \H)$ given by $\phi(X)=\rho(X)\otimes_\C 1$ for all $X\in\gog$.
\end{enumlem}
\end{lemma}

\begin{proof}
We will first show \cref{lem:conjugatereal}. We have to show that $J_\R$ is an equivariant antilinear map such that $J_\R^2=1$. We have that $J_\R$ is antilinear, since\begin{equation*}
J_\R(z(v_1,v_2))=J_\R(zv_1,\overline{z}v_2)=(\overline{z}v_2,zv_1)=\overline{z}(v_2,v_1)=\overline{z}J_\R(v_1,v_2).
\end{equation*} We have that $J_\R$ is equivariant, since\begin{align*}
J_\R\circ (\rho\oplus\overline{\rho})(X)(v_1,v_2)&=J_\R(\rho(X)v_1,\rho(X)v_2)=(\rho(X)v_2,\rho(X)v_1)\\
&=(\rho\oplus\overline{\rho})(X)\circ J_\R(v_1,v_2).
\end{align*} We have $J_\R^2(v_1,v_2)=(v_1,v_2)$, so $J_\R$ is an antiinvolution of the first kind. The fixpoint set of $J_\R$ is the real vector space $W=\{(v,v)\}\subset V\oplus \overline{V}$. We will show now that the representation $(\rho\oplus\overline{\rho})_{\restriction W}$ is equivalent to $\phi:\gog\to\End_\R(V_\R)$ given by restriction of scalars. We define $f:W\to V_\R, (v,v)\mapsto v$. This is an $\R$-linear isomorphism, since\begin{equation*}
f(r(v,v))=f(rv,\overline{r}v)=f(rv,rv)=rv.
\end{equation*} It is also equivariant, since \begin{equation*}
f\circ(\rho\oplus\overline{\rho})_{\restriction W}(X) (v,v)=f(\rho(X)v,\rho(X)v)=\rho(X)v.
\end{equation*} So $f$ is an intertwiner.

We will now show \cref{lem:conjugatequaternionic}. $J_\H=(-1,1)J_\R$ is an equivariant antilinear map by the proof of \cref{lem:conjugatereal}. We have $J_\H^2(v_1,v_2)=(-v_1,-v_2)$, so $J_\H$ is an antiinvolution of the second kind. We equip $V\oplus\overline{V}$ with a right $\H$-module structure by setting $(v,w)\cdot j=J_\H(v,w)$. We will show now that $\rho\oplus\overline{\rho}:\gog\to \End_\H(V\oplus\overline{V})$ is equivalent to $\phi:\gog\to\End_\H(V\otimes_\C \H)$ given by $\phi(X)=\rho\otimes_\C 1$ for all $X\in\gog$. We define $f:V\oplus\overline{V}\to V\otimes_\C \H, {(v,w)\mapsto v\otimes 1+w\otimes j}$. This is an $\H$-linear isomorphism, since for $z\in \C$ we have\begin{align*}
f((v,w)z)&=f(zv,\overline{z}w)=zv\otimes 1+\overline{z}w\otimes j=v\otimes z+w\otimes \overline{z}j\\
&=v\otimes z+w\otimes jz=(v\otimes 1+w\otimes j)z=f(v,w)z,\\
f((v,w)j)&=f(-w,v)=-w\otimes 1+v\otimes j=(v\otimes 1+w\otimes j)j=f(v,w)j.
\end{align*}It is also equivariant, since \begin{align*}
f\circ \phi(X)(v,w)&=f(\rho(X)v,\rho(X)w)=\rho(X)v\otimes 1+\rho(X)w\otimes j\\
&=\phi(X)(v\otimes 1+w\otimes j)=\phi(X)f(v,w).\end{align*} So $f$ is an intertwiner.
\end{proof}

\begin{lemma}\label{lem:conjugateequivalent}
Let $\rho$ be an irreducible $\C$-representation and let $J$ be an antiinvolution on $\rho\oplus\overline{\rho}$.\begin{enumlem}
\item If $J$ is an antiinvolution of the first kind, then $(\rho\oplus\overline{\rho},J)$ is equivalent to $(\rho\oplus\overline{\rho},J_\R)$.
\item If $J$ is an antiinvolution of the second kind, then $(\rho\oplus\overline{\rho},J)$ is equivalent to $(\rho\oplus\overline{\rho},J_\H)$.
\end{enumlem}
\end{lemma}

\begin{proof}
The $\C$-representation $\rho:\gog\to \End_\C(V)$ is either self-conjugated or not. Assume first that $\rho$ is not self-conjugated, i.e., $\rho$ is not equivalent to $\overline{\rho}$. Then $\End_\C(\rho\oplus\overline{\rho})=\C\oplus\C$. So by \cref{lem:ubiquitousobservation} we have $J=(z_1,z_2)J_\R$. Since $J$ is an antiinvolution, we have \begin{equation*}
\pm (v_1,v_2)=J^2(v_1,v_2)=(z_1,z_2)J_\R(z_1,z_2)J_\R(v_1,v_2)=(z_1\overline{z_2}v_1,\overline{z_2}z_1v_2).
\end{equation*} So $z_1\overline{z_2}=\pm 1$ depending on whether $J$ is an aniinvolution of the first or second kind. If $J$ is an antiinvolution of the first kind, we have ${J=(z,\overline{z}^{-1})J_\R}$ for some $z\in \C\setminus\{0\}$. Then $J$ is equivalent to $J_\R$, since $(1,\overline{z})\in \End_\C(\rho\oplus\overline{\rho})$ is an isomorphism and\begin{align*}
(1,\overline{z})J(v_1,v_2)&=(1,\overline{z})(z,\overline{z}^{-1})J_\R(v_1,v_2)=(z,1)(v_2,v_1)\\
&=(zv_2,v_1)=J_\R(1,\overline{z})(v_1,v_2).
\end{align*} If $J$ is an antiinvolution of the second kind, then $J=(-z,\overline{z}^{-1})J_\R$ is equivalent to $(-1,1)J_\R=J_\H$ for the same reason.

Now assume that $\rho$ is self-conjugated, i.e., there is an equivariant $\C$-linear isomorphism $\phi:\overline{V}\to V$. We define \begin{equation*}
F:V\oplus\overline{V}\to V\oplus V,\ (v,w)\mapsto (v,\phi(w)).
\end{equation*} The isomorphism $F$ intertwines the representation $\rho\oplus\overline{\rho}$ with $\rho\oplus\rho$. So we get $\End_\C(\rho\oplus\overline{\rho})=F^{-1}\End_\C(\rho\oplus\rho)F=F^{-1}M_2(\C)F$. So by \cref{lem:ubiquitousobservation} we have $J=F^{-1}\begin{pmatrix}
a&b\\
c&d
\end{pmatrix}F J_\R$ for some $\begin{pmatrix}
a&b\\
c&d
\end{pmatrix}\in M_2(\C)$. A useful computation is \begin{equation*}
J_\R F^{-1}\begin{pmatrix}
a&b\\
c&d
\end{pmatrix}F=F^{-1}\begin{pmatrix}
\overline{d}& \varepsilon(\rho)\overline{c}\\
\varepsilon(\rho)\overline{b}&\overline{a}
\end{pmatrix}F J_\R.
\end{equation*}So since $J$ is an antiinvolution we have \begin{align*}
\pm (v_1,v_2)&=J^2(v_1,v_2)=F^{-1}\begin{pmatrix}
a&b\\
c&d
\end{pmatrix}F J_\R F^{-1}\begin{pmatrix}
a&b\\
c&d
\end{pmatrix}F J_\R(v_1,v_2)\\
&=F^{-1}\begin{pmatrix}
a&b\\
c&d
\end{pmatrix}F F^{-1}\begin{pmatrix}
\overline{d}& \varepsilon(\rho)\overline{c}\\
\varepsilon(\rho)\overline{b}&\overline{a}
\end{pmatrix}F J_\R J_\R (v_1,v_2)\\
&=F^{-1}\begin{pmatrix}
a&b\\
c&d
\end{pmatrix}\begin{pmatrix}
\overline{d}& \varepsilon(\rho)\overline{c}\\
\varepsilon(\rho)\overline{b}&\overline{a}
\end{pmatrix}F(v_1,v_2).\end{align*} So depending on whether $J$ is an antiinvolution of first or second kind, we have \begin{equation*}
\pm \Id=\begin{pmatrix}
a&b\\
c&d
\end{pmatrix}\begin{pmatrix}
\overline{d}& \varepsilon(\rho)\overline{c}\\
\varepsilon(\rho)\overline{b}&\overline{a}
\end{pmatrix}=\begin{pmatrix}
a\overline{d}+\varepsilon(\rho)\lvert b\rvert^2& a\varepsilon(\rho)\overline{c}+b\overline{a}\\
c\overline{d}+d\varepsilon(\rho)\overline{b}&d\overline{a}+\varepsilon(\rho)\lvert c\rvert^2
\end{pmatrix}.
\end{equation*}

Let $J$ be an antiinvolution of the first kind. Then ${(\rho\oplus\overline{\rho},J)}$ is equivalent to ${(\rho\oplus\rho,J_\R)}$ if there is an isomorphism $F^{-1}\begin{pmatrix}
e&f\\
g&h
\end{pmatrix}F\in \End_\C(\rho\oplus\overline{\rho})$ such that \begin{equation*}F^{-1}\begin{pmatrix}
e&f\\
g&h
\end{pmatrix}FJ=J_\R F^{-1}\begin{pmatrix}
e&f\\
g&h
\end{pmatrix}F.\end{equation*} We check that this condition is satisfied for every antiinvolution of the first kind $J$. \begin{enumerate}
\item Assume $a\neq 0$. Then $F^{-1}\begin{pmatrix}
1&0\\
\varepsilon(\rho)\overline{b}&\overline{a}
\end{pmatrix}F\in \End_\C(\rho\oplus\overline{\rho})$ is an isomorphism and \begin{align*}
F^{-1}\begin{pmatrix}
1&0\\
\varepsilon(\rho)\overline{b}&\overline{a}
\end{pmatrix}F J&=F^{-1}\begin{pmatrix}
1&0\\
\varepsilon(\rho)\overline{b}&\overline{a}
\end{pmatrix}F F^{-1}\begin{pmatrix}
a&b\\
c&d
\end{pmatrix}F J_\R\\
&=F^{-1}\begin{pmatrix}
a&b\\
\varepsilon(\rho)\overline{b}a+\overline{a}c&\varepsilon(\rho)\lvert b\rvert^2+\overline{a}d
\end{pmatrix}F J_\R\\
&=F^{-1}\begin{pmatrix}
a&b\\
0&1
\end{pmatrix}F J_\R\\
&=J_\R F^{-1}\begin{pmatrix}
1&0\\
\varepsilon(\rho)\overline{b}&\overline{a}
\end{pmatrix}F.
\end{align*}
\item Assume $d\neq 0$. Then $F^{-1}\begin{pmatrix}
\overline{d}&\varepsilon(\rho)\overline{c}\\
0&1
\end{pmatrix}F\in \End_\C(\rho\oplus\overline{\rho})$ is an isomorphism and \begin{align*}
F^{-1}\begin{pmatrix}
\overline{d}&\varepsilon(\rho)\overline{c}\\
0&1
\end{pmatrix}F J&=F^{-1}\begin{pmatrix}
\overline{d}&\varepsilon(\rho)\overline{c}\\
0&1
\end{pmatrix}F F^{-1}\begin{pmatrix}
a&b\\
c&d
\end{pmatrix}F J_\R\\
&=F^{-1}\begin{pmatrix}
\varepsilon(\rho)\lvert c\rvert^2&\overline{d}b+\varepsilon(\rho)\overline{c}d\\
c&d\end{pmatrix}F J_\R\\
&=F^{-1}\begin{pmatrix}
1&0\\
c&d
\end{pmatrix}F J_\R\\
&=J_\R F^{-1}\begin{pmatrix}
\overline{d}&\varepsilon(\rho)\overline{c}\\
0&1
\end{pmatrix}F.
\end{align*}
\item Assume $a=d=0$ and therefore $\varepsilon(\rho)=1$. Assume that $\overline{b}\overline{c}\neq 1$. Then ${F^{-1}\begin{pmatrix}
1&\overline{c}\\
\overline{b}&1
\end{pmatrix}F\in \End_\C(\rho\oplus\overline{\rho})}$ is an isomorphism and \begin{align*}
F^{-1}\begin{pmatrix}
1&\overline{c}\\
\overline{b}&1
\end{pmatrix}f J&=F^{-1}\begin{pmatrix}
1&\overline{c}\\
\overline{b}&1
\end{pmatrix}F F^{-1}\begin{pmatrix}
0&b\\
c&0
\end{pmatrix}F J_\R\\
&=F^{-1}\begin{pmatrix}
\lvert c\rvert^2&b\\
c&\lvert b\rvert^2\end{pmatrix}F J_\R\\
&=F^{-1}\begin{pmatrix}
1&b\\
c&1
\end{pmatrix}F J_\R\\
&=J_\R F^{-1}\begin{pmatrix}
1&\overline{c}\\
\overline{b}&1
\end{pmatrix}F.
\end{align*}
\item Assume $a=d=0$ and therefore $\varepsilon(\rho)=1$. Assume that $\overline{b}\overline{c}=1$ and $b\neq \overline{b}$. Then ${F^{-1}\begin{pmatrix}
1&1\\
\overline{b}&b
\end{pmatrix}F\in {\End_\C(\rho\oplus\overline{\rho})}}$ is an isomorphism and \begin{align*}
F^{-1}\begin{pmatrix}
1&1\\
\overline{b}&b
\end{pmatrix}F J&=F^{-1}\begin{pmatrix}
1&1\\
\overline{b}&b
\end{pmatrix}F F^{-1}\begin{pmatrix}
0&b\\
\overline{b}&0
\end{pmatrix}F J_\R\\
&=F^{-1}\begin{pmatrix}
\overline{b}&b\\
1&1\end{pmatrix}F J_\R\\
&=J_\R F^{-1}\begin{pmatrix}
1&1\\
\overline{b}&b
\end{pmatrix}F.\end{align*}
\item Assume $a=d=0$ and therefore $\varepsilon(\rho)=1$. Assume that $\overline{b}\overline{c}=1$ and $b=\overline{b}$. Then ${b\in \{-1,1\}}$. Then \begin{equation*}
F^{-1}\begin{pmatrix}
1&-i\\
1&i
\end{pmatrix}F,\ F^{-1}\begin{pmatrix}
1&i\\
-1&i
\end{pmatrix}F\in \End_\C(\rho\oplus\overline{\rho})
\end{equation*}
are isomorphisms and \begin{align*}
F^{-1}\begin{pmatrix}
1&-i\\
1&i
\end{pmatrix}FJ&=F^{-1}\begin{pmatrix}
1&-i\\
1&i
\end{pmatrix}FF^{-1}\begin{pmatrix}
0&1\\
1&0
\end{pmatrix}F J_\R\\
&=F^{-1}\begin{pmatrix}
-i&1\\
i&1\end{pmatrix}F J_\R=J_\R F^{-1}\begin{pmatrix}
1&-i\\
1&i
\end{pmatrix}F\\
F^{-1}\begin{pmatrix}
1&i\\
-1&i
\end{pmatrix}FJ&=F^{-1}\begin{pmatrix}
1&i\\
-1&i
\end{pmatrix}FF^{-1}\begin{pmatrix}
0&-1\\
-1&0
\end{pmatrix}F J_\R\\
&=F^{-1}\begin{pmatrix}
-i&-1\\
-i&1\end{pmatrix}F J_\R=J_\R F^{-1}\begin{pmatrix}
1&i\\
-1&i
\end{pmatrix}F.\end{align*}
\end{enumerate} This shows that $J$ is equivalent to $J_\R$.

If $J$ is an antiinvolution of the second kind, it is equivalent to $F^{-1}\begin{pmatrix}
-1&0\\
0&1
\end{pmatrix}F J_\R=J_\H$ for the same reasons.
\end{proof}

\begin{corollary}\label{cor:conjugateirreducible}
Let $\rho$ be an irreducible $\C$-representation of a real semisimple Lie algebra $\gog$. We can conclude the following:\begin{enumerate}
\item If $\rho$ is not self-conjugated, then $(\rho\oplus\overline{\rho},J_\R)$ and $(\rho\oplus\overline{\rho},J_\H)$ are irreducible.
\item If $\rho$ is real, then \begin{itemize}
\item $(\rho\oplus\overline{\rho},J_\R)$ is equivalent to $(\rho,J_\rho)\oplus(\rho,J_\rho)$,
\item $(\rho\oplus\overline{\rho},J_\H)$ is irreducible.
\end{itemize}
\item If $\rho$ is quaternionic, then \begin{itemize}
\item $(\rho\oplus\overline{\rho},J_\R)$ is irreducible,
\item $(\rho\oplus\overline{\rho},J_\H)$ is equivalent to $(\rho,J_\rho)\oplus(\rho,J_\rho)$.
\end{itemize}
\end{enumerate}
\end{corollary}

\begin{proof}
Let $(\rho',J')$ be an $\R$- or $\H$-subrepresentation of ${(\rho\oplus\overline{\rho},J_\H)}$ or ${(\rho\oplus\overline{\rho},J_\R)}$, respectively. Then $\rho'$ is a $\C$-subrepresentation of $\rho\oplus\overline{\rho}$. Any proper $\C$-subrepresentations of $\rho\oplus\overline{\rho}$ is equivalent to $\rho$ or $\overline{\rho}$. 

If $\rho$ is not self-conjugated, neither $\rho$ nor $\overline{\rho}$ admit any antiinvolution. So $(\rho\oplus\overline{\rho},J_\R)$ and $(\rho\oplus\overline{\rho},J_\H)$ are irreducible.
If $\rho$ is real, then $\rho\oplus\overline{\rho}=\rho\oplus\rho$ admits an antiinvolution of the first kind $J_\rho\oplus J_\rho$. By \cref{lem:conjugateequivalent} $(\rho\oplus\overline{\rho},J_\R)$ is equivalent to $(\rho,J_\rho)\oplus(\rho,J_\rho)$. Since $\rho$ does not admit an antiinvolution of the second kind by \cref{lem:oneequivalent}, $(\rho\oplus\overline{\rho},J_\H)$ is irreducible.
If $\rho$ is quaternionic, then $\rho\oplus\overline{\rho}=\rho\oplus\rho$ admits an antiinvolution of the second kind $J_\rho\oplus J_\rho$. By \cref{lem:conjugateequivalent} $(\rho\oplus\overline{\rho},J_\H)$ is equivalent to $(\rho,J_\rho)\oplus(\rho,J_\rho)$. Since $\rho$ does not admit an antiinvolution of the first kind by \cref{lem:oneequivalent}, $(\rho\oplus\overline{\rho},J_\R)$ is irreducible.
\end{proof}

Now we can conclude complete reducibility of $\H$-representations.

\begin{theorem}\label{thm:completelyreducible}
Every $\H$-representation $(\rho,J)$ of a real semisimple Lie algebra $\gog$ decomposes as a direct sum of irreducible $\H$-representations.
\end{theorem}

\begin{proof}
By \cref{lem:reduciblecandidates} any $\H$-representation $(\rho,J)$ decomposes as a direct sum of $\H$-representations $(\rho_i,J_i)$, which are of one of the following two types:\begin{enumerate}
\item $\rho_i$ is a quaternionic irreducible $\C$-representation.
\item $\rho_i=\phi\oplus \overline{\phi}$ for an irreducible $\C$-representation $\phi$.
\end{enumerate} In the first case, $(\rho_i,J_i)$ is equivalent to $(\rho_i,J_{\rho_i})$ by \cref{lem:oneequivalent}, which is irreducible by \cref{cor:oneirreducible}. In the second case, $(\rho_i,J_i)$ is equivalent to $(\phi\oplus \overline{\phi},J_\H)$ by \cref{lem:conjugateequivalent}. If $\phi$ is quaternionic, $(\phi\oplus \overline{\phi},J_\H)$ is equivalent to $(\phi,J_{\phi})\oplus (\phi,J_{\phi})$ by \cref{cor:conjugateirreducible}. So $(\rho_i,J_i)$ decomposes as a direct sum of irreducible $\H$-representations by \cref{cor:oneirreducible}. If $\phi$ is not quaternionic, then $(\rho_i,J_i)$ is irreducible by \cref{cor:conjugateirreducible}.
\end{proof}

This proof works analogously for $\R$-representations (see p. 504 in \cite{Knapp2002} for a different proof for $\R$-representations).

\begin{theorem}\label{thm:irreduciblereps}
Every irreducible $\H$-representation $(\rho,J)$ is one of the following:\begin{enumerate}
\item $\rho$ is a quaternionic irreducible representation and $J=J_\rho$.
\item $\rho=\phi\oplus\overline{\phi}$, where $\phi$ is a real or not self-conjugated irreducible representation  and $J=J_\H$.
\end{enumerate} Every irreducible $\R$-representation $(\rho,J)$ is one of the following:\begin{enumerate}
\item $\rho$ is a real irreducible representation and $J=J_\rho$.
\item $\rho=\phi\oplus\overline{\phi}$, where $\phi$ is a quaternionic or not self-conjugated irreducible representation and $J=J_\R$.
\end{enumerate}
\end{theorem}

\begin{proof}
We know by \cref{cor:oneirreducible} and \cref{cor:conjugateirreducible} that the representations given in the statement are irreducible. The proof of \cref{thm:completelyreducible} shows that every $\H$-representation (resp. $\R$-representation) decomposes as a direct sum of these representations. So they are all irreducible $\H$-representation (resp. $\R$-representation).
\end{proof}

We also get a posteriori a reduction from $\H$- and $\R$-representations to quaternionic and real representations.

\begin{theorem}\label{thm:RHequivalentiffCequivalent}
A quaternionic (resp. real) representation $\rho$ admits a unique antiinvolution of second (resp. first) kind up to equivalence.
\end{theorem}

\begin{proof}
By \cref{thm:completelyreducible} it is enough to prove the statement for irreducible $\H$-representations (or $\R$-representations). By \cref{thm:irreduciblereps} we already did so proving \cref{lem:oneequivalent} and \cref{lem:conjugateequivalent}.
\end{proof}

In particular, two $\H$-representations (or $\R$-representations) $(\rho_1,J_1)$, $(\rho_2,J_2)$ are equivalent if and only if the $\C$-representations $\rho_1$, $\rho_2$ are equivalent. So it is enough to determine which irreducible $\C$-representations of a semisimple real Lie algebra $\gog$ are quaternionic to classify the irreducible $\H$-representations of $\gog$ and therefore all $\H$-representations of $\gog$. There are advanced techniques to do so in slightly different contexts (for compact groups see Chapter 26 in \cite{FultonHarris1991}, for real reductive algebraic group see \cite{Cui2016}). For our purposes the following is enough.

\subsection*{Quaternionic representations}

We reduce to the case of a simple Lie algebra $\gog$.

\begin{lemma}[Lemma 6 in \cite{Iwahori1959}]\label{lem:Iwahori6}
Let $\rho$ be an irreducible representation of a real semisimple Lie algebra $\gog_1\oplus\dots\oplus\gog_k$. We denote by $\rho_i$ the induced irreducible representation of $\gog_i$. Then $\rho$ is self-conjugated if $\rho_i$ is self-conjugated for all $i=1\dots k$. We have $\varepsilon(\rho)=\varepsilon(\rho_1)\dots \varepsilon(\rho_k)$.
\end{lemma}

So from now on we can assume that $\gog$ is a real simple Lie algebra. We want to pass from ($\R$-linear) representations of the real simple Lie algebra $\gog$ on complex vector spaces to ($\C$-linear) representations of its complexification ${\gog_\C}$. Restriction in one and complex linear extension in the other direction give us a one to one correspondence between representations of $\gog$ and its complexification $\gog_\C$. This correspondence preserves irreducibility. For the complex semisimple Lie algebra $\gog_\C$ we choose simple roots $(\alpha_i)_{i=1}^n$. This gives us a partial ordering on the space of weights of $\gog_\C$. Every irreducible representation has a highest weight. We denote the irreducible representation of highest weight $\lambda$ by $\rho(\lambda)$. Let $(\omega_i)_{i=1}^n$ be the fundamental weights corresponding to our choice of simple roots. We abbreviate $\rho_i:=\rho(\omega_i)$. Any dominant integral weight can be written uniquely as a non-negative integer combination of fundamental weights $\lambda=\sum_{i=1}^n \lambda_i\omega_i$. While the correspondence between representations of $\gog$ and ${\gog_\C}$ preserves irreducibility, it does not even make sense to talk about conjugated, real or quaternionic representations of a complex Lie algebra. So we have to keep a real form $\gog$ in mind.

\begin{definition}
Let $\gog$ be a real form of a complex Lie algebra $\gog_\C$. The \emph{$\gog$-conjugate} $\overline{\rho}^{\gog}$ of a representation $\rho$ of $\gog_\C$ is the complex linear extension of the conjugate of ${\rho}_{\restriction \gog}$. A representation $\rho$ of $\gog_\C$ is\begin{enumerate}
\item \emph{$\gog$-self-conjugated} if $\rho_{\restriction \gog}$ is self-conjugated.
\item \emph{$\gog$-real} if $\rho_{\restriction \gog}$ is real.
\item \emph{$\gog$-quaternionic} if $\rho_{\restriction \gog}$ is quaternionic.
\end{enumerate}
\end{definition}

An irreducible representation of $\gog_\C$ is either $\gog$-real, $\gog$-quaternionic or not $\gog$-self-conjugated by \cref{lem:oneequivalent}. Note that a representation can be both $\gog$-real and $\gog$-quaternionic if it is not irreducible: For example the complex linear extension of $\rho\oplus\overline{\rho}$ for any $\C$-representation $\rho$ of $\gog$ is both $\gog$-real and $\gog$-quaternionic by \cref{lem:conjugateantiinvolutions}.

We will now characterize $\gog$-quaternionic representations in terms of their highest weight. Let $\gog$ be the real form of $\gog_\C$ given by the involution $\sigma$. Let $\rho(\lambda)$ be an irreducible representation of $\gog_\C$ with highest weight $\lambda$. Then the $\gog$-conjugate of $\rho(\lambda)$ is an irreducible representation of $\gog_\C$ with some highest weight $\lambda'$.

\begin{definition}\label{def:twistedgalois}
We define the \emph{twisted Galois action} $t(\sigma)$ on the dominant integral weights of $\gog_\C$ by $t(\sigma)\lambda=\lambda'$.
\end{definition}

This action restricts to fundamental weights (see p.75-76 in \cite{Iwahori1959}), so we can write $t(\sigma)i=i'$ if $t(\sigma)\omega_i=\omega_{i'}$. The involution $t(\sigma)$ is depicted by arrows in the Satake-Tits diagram, it is denoted by $\sigma^*$ in \cite{Tits1966}. It is usually defined in a more general context, for a helpful exposition see Chapter 2 of \cite{Schoeneberg2017}.

Let $R(\gog)$ (resp. $C(\gog)$, $Q(\gog)$) be the set of indices $i$ such that $\rho_i$ is $\gog$-real (resp. not $\gog$-self-conjugated, $\gog$-quaternionic). In our notation we can rephrase Theorem 2 in \cite{Iwahori1959} as follows:

\begin{theorem}[Theorem 2 in \cite{Iwahori1959}]\label{thm:Iwahori2}
The representation $\rho(\lambda)$ of highest weight $\lambda=\sum_{i=1}^n \lambda_i\omega_i$ is \begin{enumthm}
\item\label{lem:realfundamental} $\gog$-real if and only if $\lambda_i=\lambda_{t(\sigma)i}$ for $i\in C(\gog)$ and $\sum_{i\in Q(\gog)} \lambda_i$ is even.
\item\label{lem:complexfundamental} not $\gog$-self-conjugated if and only if $\lambda_i\neq \lambda_{t(\sigma)i}$ for some $i\in C(\gog)$.
\item\label{lem:quaternionfundamental} $\gog$-quaternionic if and only if $\lambda_i=\lambda_{t(\sigma)i}$ for $i\in C(\gog)$ and $\sum_{i\in Q(\gog)} \lambda_i$ is odd.
\end{enumthm}
\end{theorem}

So to characterize all irreducible $\gog$-quaternionic representations for some real form $\gog$ of $\gog_\C$, it is enough to check which fundamental representations of $\gog_\C$ are $\gog$-real, which not $\gog$-self-conjugated and which $\gog$-quaternionic. This was done by Tits in \cite{Tits1967} (see also Table 5 in \cite{Onishchik2004}) for all simple Lie algebras $\gog$. 

We can see which fundamental representations are $\gog$-conjugated by looking at the arrows in the Satake-Tits diagram. If the vertex $i$ is not connected to another vertex by an arrow, then $\rho_i$ is $\gog$-self-conjugated and therefore either $\gog$-real or $\gog$-quaternionic. We extract a complete list of $\gog$-quaternionic fundamental representations of simple $\gog_\C$ in \cref{table:quatfundrep} as they are the most important for us. If $\gog$ is not mentioned in \cref{table:quatfundrep}, there are no irreducible $\gog$-quaternionic representations.

\begin{table}[!htbp]\begin{adjustbox}{center}
\begin{tabular}{ |c|c|c|c|c| } \hline
$\gog_\C$ & $\gog$ & Satake-Tits diagram & $Q(\gog)$\\
\hline
$A_n$& \makecell[c]{$\su(n+1-r,r)$\\ for $n$ odd, $0\leq r\leq \frac{n+1}{2}$, \\ $\frac{n+1}{2}+r$ odd} & \adjustbox{valign=m}{\begin{dynkinDiagram}[labels*={1,r,r+1,n-r-1,n-r,n},edge length=1cm]{A}{o.o*.*o.o}\invol{1}{6}\invol{2}{5}\end{dynkinDiagram}}& $\frac{n+1}{2}$\\
$A_n$& \makecell[c]{$\sl_{\frac{n+1}{2}}(\H)$\\ for $n\geq 3$ odd}\rule[-18pt]{0pt}{18pt} & \adjustbox{valign=m}{\begin{dynkinDiagram}[labels*={1,2,3,n-1,n},edge length=.75cm]{A}{*o*.o*}\end{dynkinDiagram}} &$i$ odd\\
$B_n$& \makecell[c]{$\so(2n+1-r,r)$\\ for $0\leq r\leq n$,\\ $n-r=1 \text{ or } 2 \mod 4$}\rule[-28pt]{0pt}{28pt} & \adjustbox{valign=m}{\begin{dynkinDiagram}[labels*={1,r,r+1,n-1,n},edge length=.75cm]{B}{o.o*.**}\end{dynkinDiagram}}& $n$\\
$C_n$& \makecell[c]{$\sp(n-r,r)$\\ for $0\leq r<\frac{n}{2}$}\rule[-18pt]{0pt}{18pt} &\adjustbox{valign=m}{\begin{dynkinDiagram}[labels*={1,2,3,2r,2r+1,n-1,n},edge length=.75cm]{C}{*o*.o*.**}\end{dynkinDiagram}} & $i$ odd \\
$C_n$& \makecell[c]{$\sp(\frac{n}{2},\frac{n}{2})$\\ for $n$ even} &\adjustbox{valign=m}{\begin{dynkinDiagram}[labels*={1,2,3,n-2,n-1,n},edge length=.75cm]{C}{*o*.o*o}\end{dynkinDiagram}} & $i$ odd \\
$D_n$ & \makecell[c]{$\so(2n-r,r)$\\ for $0\leq r\leq n$,\\ $n-r\equiv 2 \mod 4$} &\adjustbox{valign=m}{\begin{dynkinDiagram}[labels={1,2,r,r+1,n-2,n-1,n}, edge length=.75cm]{D}{oo.o*.***}\end{dynkinDiagram}}& $n-1$, $n$\\
$D_n$ & \makecell[c]{$\so^*_{2n}$\\ for $n\geq 6$ even} &\adjustbox{valign=m}{\begin{dynkinDiagram}[labels*={,,,,n-2}, labels={1,2,3,n-3,,n-1,n}, edge length=.75cm]{D}{*o*.*o*o}\end{dynkinDiagram}}& $i$ odd\\
$D_n$ & \makecell[c]{$\so^*_{2n}$\\ for $n\geq 5$ odd} &\adjustbox{valign=m}{\begin{dynkinDiagram}[labels*={,,,,n-2}, labels={1,2,3,n-3,,n-1,n}, edge length=.75cm]{D}{*o*.o*oo}\invold{6}{7}\end{dynkinDiagram}}& $i\neq n$ odd\\
$E_7$ & $\goe_{7(-5)}$ & \adjustbox{valign=m}{\begin{dynkinDiagram}[labels*={6,7,5,4,3,2,1}, backwards, edge length=.75cm]{E}{o*oo*o*}\end{dynkinDiagram}}& $1$,$3$,$7$\\
$E_7$ & $\goe_{7(-133)}$ & \adjustbox{valign=m}{\begin{dynkinDiagram}[labels*={6,7,5,4,3,2,1}, backwards, edge length=.75cm]{E}{*******}\end{dynkinDiagram}}& $1$,$3$,$7$\\
\hline
\end{tabular}\end{adjustbox}
\caption{$\gog$-quaternionic fundamental representations\label{table:quatfundrep}}
\end{table}

\begin{example}\label{ex:Hreps}
Let $\gog=\sl(3,\R)\oplus\su(3,1)$. We give all $\H$-representations of $\gog$. Reversing the story developed in this section, we build the $\H$-representation up from irreducible pieces instead of decomposing a general $\H$-representation.

The Lie algebra $\sl(3,\R)$ is a real form of $A_2$. Its Satake-Tits diagram is \begin{equation*}
\begin{dynkinDiagram}[labels*={1,2},edge length=1cm]{A}{oo}\end{dynkinDiagram},
\end{equation*} so $\rho_1$ and $\rho_2$ are $\sl(3,\R)$-self-conjugated. Since $\sl(3,\R)$ does not appear in \cref{table:quatfundrep}, $\rho_1$ and $\rho_2$ are $\sl(3,\R)$-real. So all irreducible representation of $A_2$ are $\sl(3,\R)$-real by \cref{thm:Iwahori2}. The Lie algebra $\su(3,1)$ is a real form of $A_3$. Its Satake-Tits diagram is 
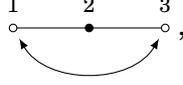
\begin{equation*}
\begin{dynkinDiagram}[labels*={1,2,3},edge length=1cm]{A}{o*o}\invol{1}{3}\end{dynkinDiagram},
\end{equation*} so $\rho_1$ and $\rho_3$ are $\su(3,1)$-conjugated. We can read off \cref{table:quatfundrep} that $\rho_2$ is a $\su(3,1)$-quaternionic representation of $A_3$. So by \cref{thm:Iwahori2} the irreducible $\su(3,1)$-quaternionic representations of $A_3$ are those of highest weight $a\omega_2+b(\omega_1+\omega_3)$ for $a$ odd.

Any irreducible $\C$-representation $\phi$ of $\gog=\sl(3,\R)\oplus\su(3,1)$ is a tensor product ${\phi=\phi_1\otimes\phi_2}$ of an irreducible $\C$-representation $\phi_1$ of $\sl(3,\R)$ and an irreducible $\C$-representation $\phi_2$ of $\su(3,1)$. By \cref{lem:Iwahori6} $\varepsilon(\phi)=\varepsilon(\phi_1)\cdot \varepsilon(\phi_2)$. So the quaternionic irreducible $\C$-representations $\phi$ of $\gog$ are of the form $\phi=\phi_1\otimes\phi_2$ for an irreducible $\C$-representation $\phi_1$ of $\sl(3,\R)$ and an irreducible $\C$-representation $\phi_2$ of $\su(3,1)$ that has highest weight $a\omega_2+b(\omega_1+\omega_3)$ for $a$ odd. By \cref{thm:completelyreducible} an irreducible $\H$-representations of $\gog$ is either $(\phi,J_\phi)$ for $\phi$ as above or $(\phi\oplus\overline{\phi},J_\H)$ for ${\phi=\phi_1\otimes\phi_2}$, where $\phi_1$ is an irreducible $\C$-representation of $\sl(3,\R)$ and $\phi_2$ is an irreducible $\C$-representation of $\su(3,1)$ with highest weight $a\omega_2+b\omega_1+c\omega_3$ for $b\neq c$ or $a$ even. By \cref{thm:completelyreducible} any $\H$-representation of $\gog$ is a direct sum of these irreducible $\H$-representations.
\end{example}

\section{Applications}\label{sec:apps}

Now we can use our knowledge about quaternionic representations to get a lower bound on the dimension of a connected manifold with invariant integrable almost quaternionic structure. We denote by $\rho_{\min,\gog}$ the $\gog$-quaternionic representation of minimal dimension of $\gog_\C$. In other words, $\rho_{\min,\gog}$ is the linear extension of the quaternionic representation of $\gog$ of minimal dimension.

\begin{corollary}\label{cor:necessarycondition2}
Let $G$ act non-trivially by quaternionic automorphisms on a connected integrable quaternionic manifold $M$ of dimension $4n$. Then $\dim_\R(\rho_{\min,\gog})\leq 4n+4=\dim(M)+4$.
\end{corollary}

\begin{proof}
By \cref{thm:necessarycondition} there is a non-trivial representation $\rho:\gog \to \sl_{n+1}(\H)$. This representation is a $\H$-representation of real dimension $4n+4$. By definition of $\rho_{\min,\gog}$ we get $\dim_\R(\rho_{\min,\gog})\leq 4n+4$.
\end{proof}

\begin{theorem}\label{thm:listsminimalquatreps}
The $\gog$-quaternionic representations of minimal dimension $\rho_{\min,\gog}$ for all real simple Lie algebras $\gog$ are given in \cref{table:minimalquatreps}.
\end{theorem}

\begin{table}[!htbp]\begin{adjustbox}{center}
\begin{tabular}{ |c|c|c|c|c| } \hline
$\gog_\C$ & $\dim_\R(\gog)$ & $\gog$ & $\rho_{\min,\gog}$ & $\dim_\R(\rho_{\min,\gog})$\\
\hline
$A_1$ & $3$ & $\su(2)$ & $\rho_1$& $4$\\
$A_3$ & $15$ & $\su(3,1)$ & $\rho_2$& $12$\\
$A_n$, $n\geq 3$ odd & $n^2+2n$ & $\sl_{\frac{n+1}{2}}(\H)$ & $\rho_1, \rho_n$ & $2(n+1)$\\
$A_n$ & $n^2+2n$ & any other $\gog$ & $\rho_1\oplus\overline{\rho_1}^{\gog}, \rho_n\oplus\overline{\rho_n}^{\gog}$ & $4(n+1)$\\
$B_2$ & $10$ & $\so(4,1)$ & $\rho_2$ & $8$\\
$B_2$ & $10$ & $\so(5)$ & $\rho_2$& $8$\\
$B_3$ & $21$ & $\so(5,2)$ & $\rho_3$& $16$\\
$B_3$ & $21$ & $\so(6,1)$ & $\rho_3$& $16$\\
$B_4$ & $36$ & $\so(6,3)$ & $\rho_4$& $32$\\
$B_4$ & $36$ & $\so(7,2)$ & $\rho_4$& $32$\\
$B_n$ & $2n^2+n$& any other $\gog$ & $\rho_1\oplus\rho_1$ & $4(2n+1)$\\
$C_n$ & $2n^2+n$& $\sp(n-r,r)$ for $0\leq r\leq \frac{n}{2}$ & $\rho_1$ & $4n$\\
$C_n$ & $2n^2+n$& any other $\gog$ & $\rho_1\oplus\rho_1$ & $8n$\\
$D_4$ & $28$ & $\so(6,2)$ & $\rho_3,\rho_4$ & $16$\\
$D_5$ & $45$ & $\so(7,3)$ & $\rho_4,\rho_5$ & $32$\\
$D_n$, $n\geq 5$ & $2n^2-n$ & $\so^*_{2n}$ & $\rho_1$ & $4n$\\
$D_n$ & $2n^2-n$ & any other $\gog$ & $\rho_1\oplus\rho_1$& $8n$\\
$E_6$ & $78$ & any $\gog$ & $\rho_1\oplus\rho_1$, $\rho_5\oplus\rho_5$ & $108$\\
$E_7$ & $133$ & $\goe_{7(-5)}$ & $\rho_1$ & $112$\\
$E_7$ & $133$ & $\goe_{7(-133)}$ & $\rho_1$ & $112$\\
$E_7$ & $133$ & any other $\gog$ & $\rho_1\oplus\rho_1$ & $224$\\
$E_8$ & $248$ & any $\gog$ & $\rho_1\oplus\rho_1$ & $992$\\
$F_4$ & $52$ & any $\gog$ & $\rho_1\oplus\rho_1$ & $104$\\
$G_2$ & $14$ & any $\gog$ & $\rho_1\oplus\rho_1$ & $28$\\
\hline
\end{tabular}\end{adjustbox}
\caption{$\gog$-quaternionic representations of minimal dimension\label{table:minimalquatreps}}
\end{table}

\begin{proof}
By \cref{thm:completelyreducible} and \cref{thm:irreduciblereps} $\rho_{\min,\gog}$ has to be either an irreducible $\gog$-quaternionic representation or $\rho_{\min,\gog}=\phi\oplus\overline{\phi}^{\gog}$, where $\phi$ is a $\gog$-real or not $\gog$-self-conjugated irreducible representation. Using \cref{thm:Iwahori2} and \cref{table:quatfundrep} we know which irreducible representations are $\gog$-quaternionic. We can compute the dimension of the representations with the Weyl dimension formula. (One can find the dimensions of the fundamental representations in many sources, eg. \cite{Tits1967} or \cite{Onishchik2004}.)

We will give the proof for $\gog=\su(3,1)$ as an example; the proof in the other cases is analogous. $\su(3,1)$ is a real form of $A_3$. We saw in \cref{ex:Hreps} that $\rho_2$ is $\su(3,1)$-quaternionic and that $\rho_1$ and $\rho_3$ are $\su(3,1)$-conjugated. So the irreducible $\su(3,1)$-quaternionic representations of minimal dimension is $\rho_2$, which has real dimension $12$. The representation $\rho_1\oplus \rho_3$ is the representation of minimal dimension of form $\phi\oplus\overline{\phi}^{\gog}$, where $\phi$ is a $\gog$-real or not $\gog$-self-conjugated irreducible representation. The real dimension of $\rho_1\oplus \rho_3$ is $16$. So $\rho_{\min,\su(3,1)}=\rho_2$.
\end{proof}

\cref{cor:necessarycondition2} and \cref{thm:listsminimalquatreps} give us a lower bound on the dimension of connected homogeneous space with $G$-invariant integrable almost quaternionic structure. We will now apply this.

\exceptional

\begin{proof}
Assume otherwise. If a Lie group $G$ acts transitively on a manifold $X$, then $\dim(X)\leq \dim_\R(G)$. So by \cref{cor:necessarycondition2} we get $\dim_\R(\rho_{\min,\gog})\leq \dim(X)+4\leq \dim_\C(\gog_\C)+4$. Using \cref{thm:listsminimalquatreps} we can check that this is not satisfied for any real form of $E_6$, $E_8$, $F_4$ or $G_2$.
\end{proof}

We will now prove \cref{thm:symmetricintegrable}.

\symmetricintegrable

\begin{proof}
Let $G/K$ be a quaternion-K\"ahler symmetric space. If $G/K$ is compact, it is integrable if and only if it is quaternionic projective space by \cref{cor:compactcase}. If $G/K$ is flat, it is quaternionic vector space $\H^n$. So without loss of generality we can assume that $G/K$ is of non-compact type. If $G/K$ has a $G$-invariant integrable almost quaternionic structure, then $\dim_\R(\rho_{\min,\gog})\leq \dim(G/K)+4$ by \cref{cor:necessarycondition2}. We proceed case-by-case using the list of quaternion-K\'ahler symmetric spaces of non-compact type in \cref{table:noncompqKsymspace}.\begin{enumerate}
\item Assume $\gog_\C=A_n$ for $n\geq 2$. Then $\gog=\su(n-1,2)$ and $\dim(G/K)=4(n-1)$. Using \cref{thm:listsminimalquatreps} and the Satake-Tits diagram of $\su(n-1,2)$ we see that the $\su(n-1,2)$-quaternionic representation of minimal dimension is $\rho_1\oplus\rho_n$, which has dimension $4(n+1)$. But $4(n+1)>4(n-1)+4$.
\item Assume $\gog_\C=B_n$ for $n\geq 3$. Then $\gog=\so(2n-3,4)$ and $\dim(G/K)=4(2n-3)$. Using \cref{thm:listsminimalquatreps} we see that the $\so(2n-3,4)$-quaternionic representation of minimal dimension is $\rho_1\oplus\rho_1$, which has dimension $4(2n+1)$. But $4(2n+1)>4(2n-3)+4$.
\item Assume $\gog_\C=B_2$ or $\gog_\C=C_n$. Then $G/K$ is quaternionic hyperbolic space and integrable by \cref{ex:quaternionhyperbolicspace}. Note that the $\sp(n-1,1)$-quaternionic representation of minimal dimension $\rho_1$ actually comes from the action of $\Sp(n-1,1)$ on quaternionic hyperbolic space.
\item Assume $\gog_\C=D_n$. Then $\gog=\so(2n-4,4)$ and $\dim(G/K)=4(2n-4)$. Using \cref{thm:listsminimalquatreps} we see that the $\so(2n-4,4)$-quaternionic representation of minimal dimension is $\rho_1\oplus\rho_1$, which has dimension $8n$. But $8n>4(2n-4)+4$.
\item Assume $\gog_\C=E_7$. Then $\gog=\goe_{7(-5)}$ and $\dim(G/K)=64$. Using \cref{thm:listsminimalquatreps} we see that the $\goe_{7(-5)}$-quaternionic representation of minimal dimension is $\rho_1$, which has dimension $112$. But $112>64+4$.
\end{enumerate}
We do not have to check for $\gog_\C\in\{E_6,E_8,F_4,G_2\}$ by \cref{cor:exceptional}. So if $G/K$ is of non-compact type, it is integrable if and only if it is quaternionic hyperbolic space $H^n_\H$.
\end{proof}

\begin{corollary}
A Riemannian symmetric space $G/K$ of dimension $4n$ for $n\geq 2$ has a $G$-invariant integrable almost quaternionic structure if and only if it is quaternionic vector space $\H^n$, quaternionic hyperbolic space $H^n_\H$ or quaternionic projective space $\mathbb{P}(\H^{n+1})$.
\end{corollary}

We stress again that our arguments did not use the Riemannian structure of quaternion-K\"ahler symmetric spaces at all. We only focused in our investigation on quaternion-K\"ahler symmetric spaces because they are the most interesting and most easily accessible $G$-homogeneous spaces with $G$-invariant almost quaternionic structure. Our method is strong enough to prove non-integrability in other cases. We give one more example.

\affineclassical

\begin{proof}
Assume otherwise. By \cref{cor:exceptional}, $\gog$ has to be a real form of $E_7$. We get as above that $\dim_\R(\rho_{\min,\gog})\leq \dim_\R(\gog)-\dim_\R(\goh)+4$. Using \cref{thm:listsminimalquatreps} we see that this can not be satisfied for $\gog=\mathfrak{e}_{7(7)}$ and $\gog=\mathfrak{e}_{7(-25)}$. The irreducible affine symmetric spaces $G/H$ with $\gog=\mathfrak{e}_{7(-133)}$ are Riemannian, since $E_{7(-133)}$ is compact. So we handled this case in \cref{thm:symmetricintegrable}. It is left to check $G/H$ with $\gog=\mathfrak{e}_{7(-5)}$ given in Table 2 in \cite{Berger1957}:\begin{itemize}
\item For $\goh=\mathfrak{so}(12)\oplus\mathfrak{sp}(1)$, we get $(\gog,\goh)$ is Riemannian.
\item For $\goh=\mathfrak{e}_{6(-14)}\oplus\R$, we get $112>133-78-1+4=58$.
\item For $\goh=\mathfrak{so}(8,4)\oplus\mathfrak{su}(2)$, we get $112>133-66-3+4=68$.
\item For $\goh=\mathfrak{su}(4,4)$, we get $112>133-63+4=74$.
\item For $\goh=\mathfrak{su}(6,2)$, we get $112>133-63+4=74$.
\item For $\goh=\mathfrak{e}_{6(2)}\oplus\R$, we get $112>133-78-1+4=58$.
\item For $\goh=\mathfrak{so}^*_{12}\oplus\mathfrak{sl}(2,\R)$, we get $112>133-66-3+4=68$.
\end{itemize} This finishes the proof.
\end{proof}

\end{document}